\documentclass[a4paper,10pt]{article}
\usepackage{amsmath,amsthm,amssymb,amscd,epsfig,esint, mathrsfs,graphics,color}
\usepackage{verbatim}
\usepackage[english]{babel}
\newtheorem{thm}{Theorem}
\newtheorem{lem}[thm]{Lemma}
\newtheorem{prop}[thm]{Proposition}
\newtheorem{coro}[thm]{Corollary}

\newtheorem{rem}[thm]{Remark}

\def\R{\mathbb R}

\def\N{\mathbb N}

\def\cC{\mathcal C}

\def\div{\operatorname{div }}

\def\supp{\operatorname{supp}}

\def\lz{\lambda}

\def\ez{\epsilon}

\def\Exp{\operatorname{Exp}}

\title{\sc{Linear transport equations for vector fields \\
with subexponentially integrable divergence}
\footnotetext{\hspace{-0.35cm}
2010 \emph{Mathematics Subject Classification}. Primary 35F05; Secondary 35F10.
\endgraf
{\it Key words and phrases}. transport equation, uniqueness, stability, divergence, $BV$-vector fields
}}
\author{\it Albert Clop, Renjin Jiang, Joan Mateu and Joan Orobitg}
\date{}

\begin{document}
\maketitle

\begin{abstract}We face the well-posedness of linear transport Cauchy problems
$$\begin{cases}\dfrac{\partial u}{\partial t} + b\cdot\nabla u + c\,u = f&(0,T)\times\R^n\\u(0,\cdot)=u_0\in L^\infty&\R^n\end{cases}$$
under borderline integrability assumptions on the divergence of the velocity field $b$. For $W^{1,1}_{loc}$ vector fields $b$ satisfying $\frac{|b(x,t)|}{1+|x|}\in L^1(0,T; L^1)+L^1(0,T; L^\infty)$ and
$$\div b\in  L^1(0,T; L^\infty) + L^1\left(0,T; \Exp\left(\frac{L}{\log L}\right)\right),$$
we prove existence and uniqueness of weak solutions. Moreover, optimality is shown in the following way: for every $\gamma>1$, we construct an example of a bounded autonomous velocity field $b$ with
$$\div(b)\in \Exp\left(\frac{L}{\log^\gamma L}\right)$$
for which the associate Cauchy problem for the transport equation admits infinitely many solutions. Stability questions and further extensions to the $BV$ setting are also addressed.
\end{abstract}

\section{Introduction}

\noindent
In this paper, we are concerned with the well-posedness (ill-posedness) of the Cauchy problem
of the transport equation
\begin{equation}\label{tran-eq}
\begin{cases}
\dfrac{\partial u}{\partial t}+b\cdot \nabla u=0 & (0,T)\times\R^n,\\
      u(0,\cdot)=u_0 &  \R^n.
   \end{cases}\end{equation}
Here $b\in L^1(0,T;W^{1,1}_{loc})$ or $b\in L^1(0,T;BV_{loc} )$,  and $u_0\in L^\infty$. A function $u\in L^1(0,T;L^1_{loc})$ is called a \textit{weak solution} to \eqref{tran-eq} if for each $\varphi\in \cC^\infty([0,T]\times\R^n)$ with compact support in  $[0,T)\times \R^n$ it holds that
\begin{equation*}
-\int_0^T\int_{\R^n} u\,\dfrac{\partial \varphi}{\partial t}\,dx\,dt-\int_{\R^n} u_0 \,\varphi(0,\cdot)\,dx-\int_0^T\int_{\R^n} u\,\div (b\,\varphi)\,dx\,dt=0.
\end{equation*}
We also say that the problem \eqref{tran-eq} is \emph{well posed} in $L^\infty(0,T;L^\infty)$ if weak solutions exist and are unique, for any $u\in L^\infty$. \\
\\
The classical method of characteristics describes, under enough smoothness of the velocity field $b$, the unique solution to \eqref{tran-eq} as the composition $u(t,x)=u_0(X(t,x))$, where $X(t,x)$ is the unique solution to the ODE
\begin{equation}\label{ODE}
\begin{cases}
\dfrac{d}{dt}\,X(t,x)=-b(t,X(t,x)),
\\X(0,x)=x.
\end{cases}
\end{equation}
When there is no smoothness, solutions of \eqref{ODE} are more delicate to understand. In the seminal work \cite{dl89}, DiPerna and Lions showed that for $b\in L^1(0,T;W^{1,1}_{loc})$ satisfying
\begin{equation}\label{hyp-b-1}
\frac{|b(t,x)|}{1+|x|}\in L^1(0,T;L^1)+L^1(0,T;L^\infty),
\end{equation}
and
$$\div b\in L^1(0,T;L^\infty),$$
the problem \eqref{tran-eq} is well-posed in $L^\infty(0,T;L^\infty)$. Moreover, the solution is renormalizable, i.e., for each $\beta\in \cC^1(\R)$,
$\beta(u)$ is the unique weak solution to the Cauchy problem
\begin{equation}\label{tran-eq-c}
\begin{cases}
\dfrac{\partial }{\partial t}\beta(u)+b\cdot \nabla \beta(u)=0& (0,T)\times \R^n,\\
      \beta(u)(0,\cdot)=\beta(u_0) & \R^n.
      \end{cases}
\end{equation}
Since that, the problem has been found many applications and has been generalized into different settings,
let us mention a few below. In \cite{d96}, Desjardins showed results of existence and uniqueness for linear
transport equations with discontinuous coefficients and velocity field having exponentially integrable divergence.
In a breakthrough paper, Ambrosio \cite{a04} extended the renormalization property to the setting of bounded
variation (or $BV$) vector fields. Cipriano-Cruzeiro \cite{cc05} found nice solutions of \eqref{ODE} for
vector fields with exponentially integrable divergence in the setting of Euclidean spaces equipped with Gauss
measures. Recently, Mucha \cite{mu10} established well-posedness for \eqref{tran-eq} with divergence of the
velocity field in $BMO$ with compact support. Also, Colombo, Crippa  and Spirito obtained at \cite{ccs}
the well-posedness of the Cauchy problem for the continuity equation with a velocity field whose divergence
is in $BMO$. See also \cite{ccs14} for the same equation with an integrable damping term. In \cite{acf14},
Ambrosio, Colombo and Figalli provide an analogy with the Cauchy-Lipschitz theory, by studying maximal
flows in the spirit of DiPerna-Lions, and using only local $L^\infty$ bounds on the divergence. For more
applications and generalizations, we refer to \cite{acfs09,af09,ccr06,cdl08,cl02,fl10,su14} and references therein.\\
\\
Our primary goal in this paper is to understand to which extent the condition $\div b\in L^1(0,T;L^\infty)$
can be relaxed so that the initial value problem \eqref{tran-eq} remains being well-posed
in $L^\infty(0,T;L^\infty)$. As it was already shown by DiPerna-Lions, the assumption
\begin{equation}\label{finiteq}
\div b\in L^1(0,T;L^q), \hspace{.3cm}\text{for some }q\in(1,\infty)
\end{equation}
is \emph{not} sufficient to guarantee uniqueness of solutions $X(t)$ of \eqref{ODE}.
As a consequence, uniqueness also fails for \eqref{tran-eq} under \eqref{finiteq}.
However, there is still some room left between $L^q$ and $L^\infty$, e.g., $BMO$
or even spaces of (sub-)exponentially integrable functions.\\
\\
Mucha \cite{mu10} recently obtained well-posedness of \eqref{tran-eq} in $L^\infty(0,T;L^\infty)$ for $W^{1,1}_{loc}$ vector fields $b$ such that $\frac{|b(t,x)|}{1+|x|}\in L^1(0,T;L^1),$
\begin{equation}\label{assumptionsubko}
\div b\in L^1(0,T;BMO),\hspace{.3cm}\text{and}\hspace{.3cm}\supp (\div b)\subset B(0,R)\text{   for some }R>0.
\end{equation}
Subko \cite{su14} further generalized Mucha's result by replacing $W^{1,1}_{loc}$ by
the local class $BV_{loc}$ of vector fields with compactly supported $BMO$ divergence.
{ In \cite{ccs}, Colombo, Crippa  and Spirito obtained the well-posedness of
the Cauchy problem for the continuity equation with a velocity field whose divergence
is a sum of a bounded function and a compactly supported $BMO$ function.}
A natural question arises here:  is the restriction on the support of $\div b$ necessary?
At the first sight, one may wonder whether well-posedness holds true if $\div b\in L^1(0,T; BMO)$
without any further restriction on the support. We do not know if this is true.
\\
\\
By using the John-Nirenberg inequality from \cite{jn61}, one sees that $BMO$ functions are locally exponentially integrable. Thus, assumption \eqref{assumptionsubko} easily gives that 
$$\div \,b\in  L^1(0,T;\Exp L),$$
where $\Exp L$ denotes the Orlicz space of globally exponentially integrable functions (see Section \ref{prelim} for a definition). Nevertheless, it is worth recalling here that no restriction on the support of $\div b$ is needed to get well-posedness if global boundedness is assumed for the divergence, namely $\div b\in L^1(0,T; L^\infty)$. Therefore, it seems reasonable to investigate if the condition
$$\div \,b\in L^1(0,T;L^\infty)+ L^1(0,T;\Exp L)$$
suffices to get well-posedness. Our first result gives a positive answer to this question. Indeed, we prove that an Orlicz space even larger than $\Exp L$ is sufficient for our purpose.

\begin{thm}\label{main}
Let $T>0$. Assume that  $b\in L^1(0,T;W^{1,1}_{loc})$ satisfying \eqref{hyp-b-1} and
\begin{equation}\label{hyp-b-2}
\div b\in L^1(0,T;L^\infty)+ L^1\left(0,T;\Exp\left(\frac{L}{\log L}\right)\right).
\end{equation}
Then for each $u_0\in L^\infty$ there exists a unique weak solution $u\in L^\infty(0,T;L^\infty)$ of the transport problem \eqref{tran-eq}.
\end{thm}

\noindent
See Section \ref{prelim} for the precise definition of the Orlicz space $\Exp(\frac{L}{\log L})$.

\begin{rem}
The conclusion of Theorem \ref{main} still holds if we add reaction and source terms. Namely, in Theorem \ref{main} the same conclusion holds if we replace \eqref{tran-eq} by
$$
\begin{cases}
\dfrac{\partial u}{\partial t}+b\cdot \nabla u+cu=f & (0,T)\times \R^n \\
      u(0,\cdot)=u_0 &\R^n
 \end{cases}
$$
provided that $c,f\in L^1(0,T;L^\infty)$, and $b\in L^1(0,T;W^{1,1}_{loc})$ satisfies \eqref{hyp-b-1} and \eqref{hyp-b-2}. The proof works similarly.
\end{rem}

\begin{rem}\label{manylogs}
One can still strengthen the borderline a bit more. More precisely, well-posedness still holds if  $b\in L^1(0,T;W^{1,1}_{loc})$ satisfies \eqref{hyp-b-1} and, at the same time, \eqref{hyp-b-2} is replaced by the less restrictive condition
$$\div b\in  L^1(0,T;L^{\infty})+L^1\left(0,T;\Exp\left(\frac{L}{\log L\,\log\log L\,\dots\,\underbrace{\log\cdots\log}_{k}L} \right)\right).$$
The proof follows similarly to that of Theorem \ref{main}.
\end{rem}

\noindent
At this point, it might bring some light reminding the chain of strict inclusions
$$\Exp L \subset \Exp\left(\frac{L}{\log L}\right)\subset \Exp\left(\frac{L}{\log L\,\log\log L\,\dots\,\underbrace{\log\cdots\log}_{k}L}\right).$$
In particular, the first one explains the following corollary, which unifies DiPerna-Lions and Mucha's results.

\begin{coro}\label{cor-exp}
Let $T>0$. Assume that  $b\in L^1(0,T;W^{1,1}_{loc})$ satisfies \eqref{hyp-b-1} and
\begin{equation*}
\div b\in L^1(0,T;L^\infty)+ L^1(0,T;\Exp L).
\end{equation*}
Then for each $u_0\in L^\infty$, there exists a unique weak solution $u\in L^\infty(0,T;L^\infty)$ of the Cauchy problem \eqref{tran-eq}.
\end{coro}
\noindent
The proof of Theorem \ref{main} will be built upon the renormalization property by DiPerna-Lions \cite{dl89} and properties of Orlicz spaces. A key ingredient is an a priori estimate by using the backward equation, which shows that if
$u\in L^\infty(0,T;L^\infty)$ is a solution of \eqref{tran-eq} with the initial value $u_0\equiv 0$, then $u\in L^\infty(0,T;L^2\cap L^\infty)$. See Proposition \ref{3-energy} below.
Indeed, the idea behind this is a kind of multiplicative property. That is, if $u_1, u_2 \in L^\infty(0,T;L^\infty)$ satisfy
$$
\dfrac{\partial u_i}{\partial t}+b\cdot \nabla u_i+c_i\,u=0 \ \mathrm{in }\ (0,T)\times\R^n,
$$
then the pointwise multiplication $u_1u_2$ solves
$$
\dfrac{\partial (u_1u_2)}{\partial t}+b\cdot \nabla (u_1u_2)+(c_1+c_2)\,(u_1u_2)=0 \ \mathrm{in }\ (0,T)\times\R^n.$$
See Proposition \ref{commutator-bv-2} below for the details.\\
\\
Notice that our assumption \eqref{hyp-b-2} on the divergence is too weak to guarantee the well-posedness of the transport equation in the $L^p$ case for finite values of $p$. To explain this, let us assume for a while that $b$ generates a flow $X(t)=X(t,x)$ through the ODE \eqref{ODE}. Boundedness of $\div b$ guarantees that the image $X(t)_\sharp m$ of Lebesgue measure $m$ is absolutely continuous and has bounded density (see \cite{dl89}). If $\div b$ is not bounded, but only (sub)-exponentially integrable, then one may still expect $X(t)_\sharp m<<m$, but boundedness of density might be lost. Thus no control on $L^p$ norms is expected if $p\in [1,\infty)$. \\
\\
At this point it is worth mentioning that the existence and uniqueness of such a flow $X(t)$ is not an easy issue in our context. Nevertheless, if one assumes $\frac{|b(x,t)|}{1+|x|}\in L^1(0,T; L^\infty)$ and $\div b\in L^1(0,T;L^\infty)+L^1(0,T;\Exp L)$, then a unique flow can be obtained as a consequence of the results from \cite{cc05}. We will come back to the flow issue in a forthcoming paper.\\
\\
We have the following quantitative estimate in $L^p\cap L^\infty$ case under assumption \eqref{hyp-b-2}. For an easier formulation in the case $\div b\in L^1(0,T;L^\infty)+L^1(0,T;\Exp L)$, see Corollary \ref{quant-ex}.

\begin{thm}\label{quant}
Let $T,M>0$ and $1\le p<\infty$. Suppose that $b\in L^1(0,T;W^{1,1}_{loc})$ satisfies \eqref{hyp-b-1} and
\eqref{hyp-b-2}. There exists $\epsilon>0$ such that, for every $u_0\in L^p\cap  L^\infty$ with
$\|u_0\|_{L^\infty}\le M$ and $\|u_0\|^p_{L^p}< \ez$, the transport problem \eqref{tran-eq} has a unique solution $u$ and moreover
$$
\left|\log\log\log\left( \frac{1}{\|u\|^p_{L^\infty(0,T; L^p)}}\right)-\log\log\log\left(\frac{1}{\|u_0\|^p_{L^p}}\right)\right|\leq 16e\int_0^T\beta(s)\,ds
$$
where $\div b=B_1+B_2$ and $\beta(t)=\|B_1(t,\cdot)\|_{\Exp(\frac{L}{\log L})}+\|B_2(t,\cdot)\|_{L^\infty}$.
\end{thm}

\noindent
Relying on Ambrosio's seminal result \cite{a04}, Theorem \ref{main} admits an extension to the setting of bounded variation ($BV$) vector fields.

\begin{thm}\label{main-bv}
Let $T>0$. Assume that  $b\in L^1(0,T; BV_{loc})$ satisfying \eqref{hyp-b-1} and \eqref{hyp-b-2}. Then for every $u_0\in L^\infty$ there exists a unique weak solution $u\in L^\infty(0,T;L^\infty)$ of the transport problem \eqref{tran-eq}.
\end{thm}

\noindent
Concerning the optimality of \eqref{hyp-b-2} in Theorem \ref{main}, and after re-analyzing an example from \cite[Section 4.1]{dl89}, we can show that the condition
$$b\in L^\infty(0,T;L^\infty),\hspace{.3cm}\text{and}\hspace{.3cm} \div b\in L^1\left(0,T;\Exp\left(\frac{L}{\log^\gamma L}\right)\right)\text{ for some }\gamma>1$$
is not sufficient to guarantee uniqueness.

\begin{thm}\label{counter}
Let $T>0$. Given $\gamma\in (1,\infty)$, there exists a vector field $b\in L^1(0,T;W^{1,1}_{loc})$,
that satisfies
$$b\in L^\infty(0,T;L^\infty),\hspace{.3cm}\text{and}\hspace{.3cm} \div b\in L^1\left(0,T;\Exp\left(\frac{L}{\log^\gamma L}\right)\right),$$
such that for some $u_0\in \cC^\infty_c(\R^n)$ the Cauchy problem \eqref{tran-eq} admits infinitely many weak solutions $u\in L^\infty(0,T;L^\infty)$.
\end{thm}

\noindent
The proof is based on DiPerna-Lions' example { on the non-uniqueness of the flow}
\cite[Section 4.1]{dl89}. The key points are to construct an explicit smooth function
vanishing exactly at the points of the one third Cantor set, { and to show that $u_0$ composed with the flow is
a distributional solution to the transport equation.}$^{1}$
{\footnotetext{ 1. We thank G. Crippa for pointing out  to us
 the issue of showing $u_0$ composed with the flow is a solution.}}

\begin{rem}
Similar examples to that in Theorem \ref{counter} can be found in the setting of Remark \ref{manylogs}, with a small modification on the smooth function $g$ (see Step 2 of the proof). More precisely, given $\gamma\in (1,\infty)$ and $k\in\{2,3,...\}$, one can find a vector field $b\in L^1(0,T;W^{1,1}_{loc})$ satisfying \eqref{hyp-b-1} and such that
$$\div b\in  L^1\left(0,T;\Exp\left(\frac{L}{(\log L) \,(\log\log L)\,\dots\,(\underbrace{\log\cdots\log}_{k}L)^\gamma}\right)\right),$$
for which the Cauchy problem \eqref{tran-eq} admits, for every $u_0\in\cC^\infty_c(\R^n)$, infinitely many weak solutions $u\in L^\infty(0,T;L^\infty)$. \end{rem}

\begin{rem}
In the context of Corollary \ref{cor-exp}, and arguing again as in the proof of Theorem \ref{counter}, one can show that the condition $\div b\in L^1(0,T;\Exp L)$ cannot be replaced by $\div b \in L^1(0,T; \Exp(L^{1/\gamma}))$ if $\gamma>1$.
\end{rem}

\begin{rem}
The example of Theorem \ref{counter} admits a further generalization to the setting of the Euclidean space when equipped with Gaussian measure $d\gamma_n$. Namely, one can show that the assumption
$$\exp\{C|b|+C(|\div _{\gamma_n}b|)^\alpha\}\in L^1(0,T;L^1(\gamma_n))\hspace{1cm}\text{ for some }\alpha\in(0,1)$$
does not imply uniqueness of the flow, and therefore uniqueness for \eqref{tran-eq} also fails. See \cite{cc05,af09,fl10}.
\end{rem}

\noindent
The paper is organized as follows. In Section \ref{prelim} we recall some basic aspects of Orlicz spaces, and prove some technical estimates. In Section 3 we prove Theorem \ref{main}. Section 4 is devoted to stability results. In Section 5, we prove Theorem \ref{main-bv}. In the last section we prove Theorem \ref{counter}. Throughout the paper, we denote by $C$ positive constants which are independent of the main parameters, but which may vary from line to line.

\section{Orlicz spaces}\label{prelim}

\noindent
We will need to use some Orlicz spaces and their duals. For the reader's convenience, we recall here some definitions. See the monograph \cite{rr91} for the general theory of Orlicz spaces. Let
$$P:[0,\infty)\mapsto [0,\infty),$$
be an increasing homeomorphism onto $[0,\infty)$, so that $P(0)=0$ and $\lim_{t\to\infty}P(t)=\infty$. The Orlicz space $L^P$ is the set of measurable functions $f$ for which the Luxembourg norm
$$\|f\|_{L^P}=\inf\left\{\lambda>0:\int_{\R^n}P\left(\frac{|f(x)|}{\lambda}\right)\,dx\leq 1\right\}$$
is finite. In this paper we will be mainly interested in two particular families of Orlicz spaces. Given $r,s\geq 0$, the first family corresponds to
$$P(t)=t\left(\log^+t\right)^{r}\,\left({\log^+ \log^+}t\right)^{s},$$
where $\log^+t:=\max\{1,\log t\}.$ The obtained $L^{P}$ spaces are known as \emph{Zygmund spaces}, and will be denoted from now on by $L\log^{r}L\,\log\log^{s}L $. The second family is at the upper borderline. For $\gamma\geq 0$ we set
\begin{equation}\label{psi-function}
P(t)=\exp\left\{\frac{t}{(\log^+t)^\gamma}\right\}-1, \hspace{1cm} \,t\geq 0.
\end{equation}
Then we will denote the obtained $L^P$ by $\Exp(\frac{L}{\log^\gamma L} )$. If $\gamma=0$ or $\gamma=1$, we then simply write $\Exp L$ and $\Exp(\frac{L}{\log L})$, respectively. Note that $0\leq \gamma_1< \gamma_2$ implies $\Exp L\subset\Exp(\frac{L}{\log^{\gamma_1}L} )\subset \Exp(\frac{L}{\log^{\gamma_2} L})$. Also, let us observe that compactly supported $BMO$ functions belong to $\Exp L$. Similarly, we will denote by
$$\Exp\left(\frac{L}{\log L\,\log\log L\,\dots\,(\underbrace{\log\cdots\log}_{k}L)^\gamma}\right)$$
the Orlicz space corresponding to
$$P(t)=
\exp\left\{
\frac{t}{(\log^+t)\,(\log^+\log^+ t  )\dots\,(\underbrace{\log^+\cdots \log^+}_{k} t) ^\gamma}\right\}-1, \hspace{1cm} \,t\geq 0.$$
The following technical lemma will be needed at Section 3.

\begin{lem}\label{duality}
If $f\in L\log L\log\log L $  and $g\in \Exp(\frac{L}{\log L})$  then $fg\in L^1$ and
$$\int_{\R^n}|f(x)g(x)|\,dx\le 2\|f\|_{L\log L\log\log L}\,\|g\|_{\Exp(\frac{L}{\log L})}.$$
Moreover, if $f\in L^\infty\cap L\log L\log\log L $ then
$$ \aligned
\|f\|&_{L\log L\log\log L}\\
&\leq 2e\|f\|_{L^1 }\bigg(\log(e+\|f\|_{L^\infty})+|\log(\|f\|_{L^1 })|\bigg)\,\bigg(\log\log(e^e+\|f\|_{L^\infty })+|\log|\log(\|f\|_{L^1 })||\bigg).\endaligned$$
\end{lem}
\begin{proof} We refer to \cite[p.17]{rr91} for the H\"older inequality. Towards the second estimate, we start by noting that if $f\in L\log L\log\log L$ then $f\in L^1$ and $\|f\|_{L^1}\leq \|f\|_{L\log L\log\log L}$. By setting $M=\|f\|_{L^\infty}$,
$$\lambda=\|f\|_{L^1}\left[\log(e+M)+|\log(\|f\|_{L^1})|\right]\left[\log\log(e^e+M)+|\log|\log(\|f\|_{L^1})||\right],$$
and calling
$$E=\{x\in\R^n:\, |f(x)|\le e^e\lambda\},$$
we see that
\begin{eqnarray*}
&&\int_{\R^n}\frac{|f(x)|}{\lz}\log^+
\left(\frac{|f(x)|}{\lz}\right)\log^+\log^+\left(\frac{|f(x)|}{\lz}\right)\,dx\\
&&\quad\le \frac e\lz\int_E|f(x)|\,dx+\int_{\R^n\setminus E}\frac{|f(x)|}{\lz}\log
\left(\frac{|f(x)|}{\lz}\right)\log\log\left(\frac{|f(x)|}{\lz}\right)\,dx\\
&&\quad \le e+\int_{\R^n\setminus E}\frac{|f(x)|}{\lz}\log
\left(\frac{e+M}{\|f\|_{L^1}}\right)\log\log\left(\frac{e^e+M}{\|f\|_{L^1}}\right)\,dx\\
&&\quad \le e+\int_{\R^n\setminus E}\frac{|f(x)|\log\left[\log({e^e+M})+|\log{\|f\|_{L^1}}|\right]}{\|f\|_{L^1}\left[\log\log(e^e+M)+|\log|\log(\|f\|_{L^1})||\right]}\,dx.
\end{eqnarray*}
Notice that for $x\ge e$ and $y\ge 0$, it holds that
$$\log(x+y)\le 2\log x+2|\log y|,$$
which implies that
\begin{eqnarray*}
&&\int_{\R^n}\frac{|f(x)|}{\lz}\log^+
\left(\frac{|f(x)|}{\lz}\right)\log^+\log^+\left(\frac{|f(x)|}{\lz}\right)\,dx\\
&&\quad \le e+\int_{\R^n\setminus E}\frac{2|f(x)|\left[\log\log({e^e+M})
+|\log(|\log{\|f\|_{L^1}}|)|\right]}{\|f\|_{L^1}\left[\log\log(e^e+M)+|\log|\log(\|f\|_{L^1})||\right]}\,dx\le 2e.
\end{eqnarray*}
Therefore, we see that
\begin{eqnarray*}
&&\int_{\R^n}\frac{|f(x)|}{2e\lz}\log^+
\left(\frac{|f(x)|}{2e\lz}\right)\log^+\log^+\left(\frac{|f(x)|}{2e\lz}\right)\,dx\\
&&\quad\le\int_{\R^n}\frac{|f(x)|}{2e\lz}\log^+
\left(\frac{|f(x)|}{\lz}\right)\log^+\log^+\left(\frac{|f(x)|}{\lz}\right)\,dx\le 1,
\end{eqnarray*}
which gives the desired estimate.
\end{proof}

\section{Existence and Uniqueness}\label{section2}

\noindent
The main goal of this section is proving Theorem \ref{main}. To this end, we will first prove existence and uniqueness when the initial data is in $L^\infty\cap L^p$ for some $p\in[1,\infty)$ (see Proposition \ref{p-unique} below). Later on, we will use this fact to show in Proposition \ref{3-energy} that any weak solution $u\in L^\infty(0,T;L^\infty)$ to \eqref{tran-eq} with vanishing initial data $u_0\equiv 0$ is indeed uniformly square summable, i.e. $u\in L^\infty(0,T; L^2)$. These two steps will make the proof of Theorem \ref{main} almost automatic.\\
\\
We start with an existence result for initial data in $L^p\cap L^{\infty}$, $p\in[1,\infty)$, which holds under much milder assumptions on $\div b$.

\begin{prop}\label{p-exist}
Let $p\in [1,\infty)$ and $b\in L^1(0,T;W^{1,1}_{loc})$ be such that
$$\div b\in L^1(0,T;L^1)+ L^1(0,T;L^\infty).$$
Assume also that $c\in L^1(0,T;L^\infty)$. Then for every $u_0\in L^\infty\cap L^p$, there is a weak solution $u\in L^\infty(0,T;L^p\cap L^\infty)$ to the transport problem
$$
\begin{cases}
\dfrac{\partial u}{\partial t}+b\cdot \nabla u+c\,u=0 & (0,T)\times\R^n  ,\\
      u(0,\cdot)=u_0 & \R^n.
      \end{cases}
$$
\end{prop}
\begin{proof}
We will follow the usual method of regularization. Let   $0\leq \rho\in \cC^{\infty}_c(\R^n)$ be such that $\int_{\R^n} \rho(x)\,dx=1$. For each $\ez>0$, set $\rho_\ez(x)=\frac 1{\ez^n}\rho(x/\ez)$, and
$b_\ez=b\ast\rho_\ez$, $c_\ez=c\ast\rho_\ez$, $u_{0,\ez}=u_0\ast\rho_\ez$.
Since $$\div b\in L^1(0,T;L^1)+ L^1(0,T;L^\infty),$$
we have for each $\ez>0$,
$$\div b_\ez=(\div b)\ast\rho_\ez\in L^1(0,T;L^\infty).$$
Therefore, $b_\ez$ and $c_\ez$ satisfy the requirements from DiPerna-Lions \cite[Proposition 2.1, Theorem 2.2]{dl89}.
Since $u_{0,\ez}\in L^p\cap L^\infty$, it follows that there exists a unique solution $u_\ez\in L^\infty(0,T;L^p\cap L^\infty)$ to the transport equation
$$
   \begin{cases}
    \dfrac{\partial u_\ez}{\partial t}+b_\ez\cdot \nabla u_\ez+c_\ez u_\ez=0 & (0,T)\times \R^n,\\
      u_\ez(0,\cdot)=u_{0,\ez} &   \R^n.
      \end{cases}
      $$
Moreover, we can bound $u_\ez$ in $L^\infty(0,T;L^\infty)$ as
\begin{equation}\label{infty-bdd}
\aligned
\|u_\ez\|_{L^\infty(0,T; L^\infty)}
&\le \|u_{0,\ez}\|_{L^\infty}\exp\left\{\int_0^T\|c_\ez(t)\|_{L^\infty}\,dt\right\}\\
&\le \|u_{0}\|_{L^\infty}\exp\left\{\int_0^T\|c(t)\|_{L^\infty}\,dt\right\}=:M.
\endaligned
\end{equation}
Therefore, by extracting a subsequence, $\{\ez_k\}_k$, we know that $u_{\ez_k}$ converges to some $u$ in the weak-$\ast$ topology of $L^\infty(0,T;L^\infty)$.\\
\\
Now, the smoothness allows us to say that
$$\frac{\partial |u_\ez|^p}{\partial t}+b_\ez\cdot \nabla |u_\ez|^p+p\,c_\ez \,|u_\ez|^p=0.$$
But we also know that $\int_{\R^n} \div (b_\ez |u_\ez|^p)\,dx=0$, since $b_\ez |u_\ez|^p\in W^{1,1}$. Thus, integrating on $\R^n$ we get that
\begin{equation}\label{regular-energy}
\frac{\partial}{\partial t}\int_{\R^n} |u_\ez|^p\,dx-\int_{\R^n} |u_\ez|^p\div b_\ez\,dx
+\int_{\R^n} pc_\ez |u_\ez|^p\,dx=0.
\end{equation}
Our assumptions on $\div b$ allow us to decompose $\div \,b=B_{1}+B_{2},$ where $B_1\in L^1(0,T;L^1)$ and $B_2\in L^1(0,T;L^\infty)$. By letting $B_{1,\ez}=B_1\ast \rho_\ez$ and $B_{2,\ez}=B_2\ast\rho_\ez$, we get from \eqref{regular-energy} that
$$\aligned
\|u_\ez(T)\|_{L^p}^p\leq &\, \|u_{0,\ez}\|_{L^p}^p+M^p\int_0^T\|B_{1,\ez}\|_{L^1}\,dt+\int_0^T\|B_{2,\ez}-p\,c_\ez\|_{L^\infty} \|u_\ez\|_{L^p}^p\,dt\\
\endaligned$$
We then see that
\begin{equation}\label{p-bdd}
\|u_\ez(T)\|_{L^p}^p\leq
\left\{\|u_{0}\|_{L^p}^p+M^p\int_0^T\|B_1\|_{L^1}\,dt\right\}\, \exp\left\{\int_0^T\|B_2-p\,c\|_{L^\infty}\,dt\right\},
\end{equation}
i.e., $\{u_\ez\}$ is uniformly bounded in $L^\infty(0,T;L^p)$. Hence, there exists a subsequence of $\{\ez_k\}_k$, $\{\ez_{k_j}\}_{k_j}$, such that
$u_{\ez_{k_j}}$ weakly converges to some $\tilde u\in L^\infty(0,T;L^p)$, if $p>1$.
For $p=1$, notice that, since $\{u_\ez\}\in L^\infty(0,T;L^\infty)$ with a uniform upper bound independent of $\ez$,
$\{u_\ez\}$ is weakly relative compact  in $L^\infty(0,T;L^1_{loc}).$ Therefore,
 there exists a subsequence of $\{\ez_k\}_k$, $\{\ez_{k_j}\}_{k_j}$, such that
$u_{\ez_{k_j}}$ weakly converges to some $\tilde u\in L^\infty(0,T;L^1)$.
%
By using a duality argument, it is easy to see that $u=\tilde u$ a.e.,
which  is the required solution. Moreover, from \eqref{infty-bdd} and \eqref{p-bdd}, we see that
\begin{eqnarray}\label{infty-bdd-solution}
\|u\|_{L^\infty(0,T; L^\infty)}\le \|u_{0}\|_{L^\infty}\exp\left\{\int_0^T\|c(t)\|_{L^\infty}\,dt\right\}.
\end{eqnarray}
and
\begin{eqnarray}\label{p-bdd-solution}
\|u\|_{L^\infty(0,T; L^p)}^p\leq \left\{\|u_{0}\|_{L^p}^p+M^p\int_0^T \|B_1\|_{L^1}\,dt\right\}
\exp\left\{\int_0^T\||B_2- p\,c|\|_{L^\infty}\,dt\right\},
\end{eqnarray}
which completes the proof.
\end{proof}

\noindent
The following commutator estimate is a special case of DiPerna-Lions \cite[Theorem 2.1]{dl89}.

\begin{lem}[DiPerna-Lions]\label{regularize}
Let $u\in L^\infty(0,T;L^\infty)$ be a solution to the transport equation
$$
\begin{cases}
\dfrac{\partial u}{\partial t}+b\cdot \nabla u+c\,u=0 &   (0,T)\times\R^n,\\
u(0,\cdot)=u_0 &  \R^n.
      \end{cases}
$$
Here $b\in L^1(0,T;W^{1,1}_{loc})$ and $c\in L^1(0,T;L^1_{loc})$. Let $0\le\rho\in \cC^\infty_c(\R^n)$
satisfy $\int_{\R^n}\rho\,dx=1$, and $\rho_\epsilon(x)=\ez^{-n}\rho(x/\ez)$. Then, $u_\ez=u\ast\rho_\epsilon$ satisfies
$$\frac{\partial u_\epsilon}{\partial t}+b\cdot\nabla u_\epsilon+c\,u_\epsilon=r_\epsilon,$$
where $r_\epsilon\to 0$ in $L^1(0,T;L^1_{loc})$ as $\epsilon\to 0$.
\end{lem}

\noindent
Lemma \ref{regularize} above allows us to prove uniqueness when the initial value is in $L^p\cap L^\infty$.

\begin{prop}\label{p-unique}
Let $p\in [1,\infty)$ and assume that  $b\in L^1(0,T;W^{1,1}_{loc})$ satisfies \eqref{hyp-b-1} and
\eqref{hyp-b-2}. Assume also that $c\in L^1(0,T;L^\infty)$. Then, for every $u_0\in L^\infty\cap L^p$ the Cauchy problem
$$
\begin{cases}
    \dfrac{\partial u}{\partial t}+b\cdot \nabla u+cu=0 &  (0,T)\times \R^n ,\\
      u(0,\cdot)=u_0 &  \R^n,
      \end{cases}
$$
admits a unique weak solution $u\in L^\infty(0,T;L^p\cap L^\infty)$ .
\end{prop}

\begin{proof} Since $\Exp(\frac{L}{\log L})\subset L^p$ for every $p\in [1,\infty)$, we know by Proposition \ref{p-exist} that there exists a weak solution $u\in L^\infty(0,T;L^p\cap L^\infty)$. In order to get uniqueness, it suffices to assume that $u_0\equiv 0$, because the equation is linear. We start by regularizing the Cauchy problem as in Lemma \ref{regularize} so we get a regularized problem,
$$
\begin{cases}
\dfrac{\partial u_\ez}{\partial t}+b\cdot \nabla u_\ez+cu_\ez=r_\ez &(0,T)\times \R^n ;\\
 u_\ez(0,\cdot)= 0 & \R^n.
 \end{cases}$$
Also, $r_\ez\to 0$ as $\epsilon\to 0$ in the $L^1(0,T;L^1_{loc})$ sense, by Lemma \ref{regularize}. Now, for each $R>0$, let $\psi_R\in \cC^\infty_c(\R^n)$ be a cutoff function, so that
\begin{equation}\label{cutoff}
\aligned
0\leq \psi_R\leq 1, \,\,\,&\psi_R(x)=1\text{ whenever }|x|\leq R,\\
&\psi_R(x)=0\text{ whenever }|x|\geq 2R,\,\,\,\text{and }|\nabla\psi_R(x)|\leq\frac{C}{R}.
\endaligned\end{equation}
By noticing that
\begin{equation}
\frac{\,d}{\,dt}\int_{\R^n} |u_\ez|^p\psi_R\,dx+\int_{\R^n}b\cdot\nabla |u_\ez|^p\psi_R\,dx+\int_{\R^n} pc |u_\ez|^p\psi_R\,dx=\int_{\R^n} r_\ez p|u_\ez|^{p-1}\psi_R\,dx,
\end{equation}
and integrating over time, we see that for every $t\leq T$
\begin{eqnarray*}
&&\int_{\R^n}|u_\ez(t,\cdot)|^p\,\psi_R \,dx\\
&&=\int^t_0\int_{\R^n}\left(\div  b-pc\right) |u_\ez|^p\psi_R\,dx\,ds+\int_0^t\int_{\R^n} b\cdot \nabla \psi_R|u_\ez|^p\,dx\,ds+\int_0^t\int_{\R^n} r_\ez p|u_\ez|^{p-1}\psi_R\,dx\,ds.
\end{eqnarray*}
By Lemma \ref{regularize} and dominated convergence theorem, letting $\ez\to 0$ yields
\begin{equation}\label{energy-control-3}
\int_{\R^n}|u(t,\cdot)|^p\,\psi_R\,dx=\int^t_0\int_{\R^n}\left(\div  b-pc\right) |u|^p\,\psi_R\,dx\,ds+\int_0^t\int_{\R^n} b\cdot \nabla \psi_R\,|u|^p\,dx\,ds.
\end{equation}
Using the assumption \eqref{hyp-b-1} on $b$, and the facts $|u|^p\in L^\infty(0,T;L^1\cap L^\infty)$ and $|\nabla \psi_R|\leq C/R$, one obtains
$$\lim_{R\to\infty}\left|\int_0^t\int_{\R^n} b\cdot
\nabla \psi_R|u|^p\,dx\,ds\right|\le \lim_{R\to\infty}\left|\int_0^t\int_{B(0,2R)\setminus B(0,R)}
 \frac{|b(s,x)|}{1+|x|}|u|^p\,dx\,ds\right|=0,$$
which kills the second term on the right hand side of \eqref{energy-control-3}. For the first term,  write
$$\div \,b=B_{1}+B_{2},$$
where $B_1\in L^1(0,T;\Exp(\frac{L}{\log L}))$ and $B_2\in L^1(0,T;L^\infty)$. Letting $R\to\infty$ in \eqref{energy-control-3} yields
\begin{eqnarray}\label{energy-control-4}
\int_{\R^n}|u(t,x)|^p\,dx&&\le \int^t_0\int_{\R^n}[|\div  b|+ p|c|] |u|^p\,dx\,ds\nonumber\\
&&\le \int^t_0\int_{\R^n}|B_1| |u|^p\,dx\,ds+
\int^t_0\left\|[|B_2|+pc\right\|_{L^\infty}\int_{\R^n}|u|^p\,dx\,ds
\end{eqnarray}
Recall that $B_1\in L^1(0,T;\Exp(\frac{L}{\log L}))\subset L^1(0,T;L^p)$ for each $p\in [1,\infty)$.
This, together with $u\in L^\infty(0,T;L^p\cap L^\infty)$, further implies that there exists $T_1>0$, such that
\begin{equation}\label{energy-control-1}
\int_{\R^n}|u(t,x)|^p\,dx<\exp\left\{-\exp\left\{e+\|u\|_{L^\infty(0,T; L^\infty)}\right\}\right\}
\end{equation}
for each $t\in (0,T_1)$. For convenience, in what follows we denote by $\alpha(t)$, $\beta_1(t)$, $\beta_2(t)$ the quantities $\|u(t,\cdot)\|_{L^p}^p$, $\|B_1\|_{\Exp(\frac{L}{\log L})}$ and $\||B_2|+pc\|_{L^\infty}$, respectively.
Denote $\|u\|_{L^\infty(0,T;L^\infty)}$ by $M$. From the first estimate of Lemma \ref{duality}, we find that
\begin{equation}\label{holderbd}
\int_{\R^n}|B_1| |u|^p\,dx\leq 2\|B_1\|_{\Exp(\frac{L}{\log L})}\||u|^p\|_{L\log L\log\log L}.
\end{equation}
By the second estimate of Lemma \ref{duality}, the factor $\||u|^p\|_{L\log L\log\log L}$ is bounded by
\begin{eqnarray}\label{holder-estimate}
2e\|B_1\|_{\Exp(\frac{L}{\log L})}\,\alpha(t)\bigg(\log(e+M)+|\log \alpha(t)|\bigg)
\bigg(\log\log(e^e+M)+|\log|\log( \alpha(t))||\bigg).
\end{eqnarray}
Notice that by \eqref{energy-control-1} we have
$$\log(e+M)\le \left|\log (\alpha(t))\right|=\log\frac{1}{\alpha(t)}$$
and
 $$\log\log(e^e+M)\le \left|\log\left(|\log \alpha(t)|\right)\right|=\log\log\frac 1{\alpha(t)}$$
 for $t\in (0,T_1)$. This fact, together with the inequalities \eqref{energy-control-3}, \eqref{holderbd} and \eqref{holder-estimate}, gives
\begin{equation}\label{alpha-estimate}
\aligned
\alpha(t)
&\leq 16e\int_{0}^t \beta_1(s) \alpha(s)\log \frac1{\alpha(s)}\log\log\frac1{\alpha(s)}+\beta_2(s)\alpha(s)\,ds \\
&\leq 16e\int_{0}^t \beta(s) \alpha(s)\log \frac1{\alpha(s)}\log\log\frac1{\alpha(s)}\,ds,
\endaligned
\end{equation}
where we denote by $\beta(s)$ the quantity $\beta_1(s)+\beta_2(s)$. We will now use a Gronwall type argument. For each $s\in (0,T]$, let
$$\alpha^\ast(s)=\exp\left\{-\exp\left\{\exp\left\{\log\log\log\frac 1{\epsilon}-16e\int_0^s\beta(s)\,ds\right\}\right\}\right\},$$
where $\epsilon>0$ is chosen small enough so that
$$\alpha^\ast(T)=\exp\left\{-\exp\left\{\exp\left\{\log\log\log\frac 1{\epsilon}-16e\int_0^T\beta(s)\,ds\right\}\right\}\right\}<\exp\left\{-\exp\{e+M\}\right\}.$$
From the definition, we see that $\alpha^\ast$ is Lipschitz continuous and increasing on $[0,T]$. Moreover, for every $t\in [0,T]$ it holds that
$$\alpha^\ast(t)=\epsilon+16e\int_0^t \beta(s) \alpha^\ast(s)\log \frac 1{\alpha^\ast(s)} \log\log\frac 1{\alpha^\ast(s)}\,ds.$$
Also, we see from \eqref{energy-control-1} that $\alpha(t)$ takes values on $[0,\exp\{-e^{e+M})$ if $t\in (0,T_1)$, and the function $s\mapsto s|\log s||\log(|\log s|)|$ is increasing on that interval. From this, the definition of $\alpha^\ast$ and \eqref{alpha-estimate}, we conclude that for each $t\in [0,T_1]$,
$$0\le \alpha(t)\le \alpha^\ast(t)\le \exp\left\{-\exp\left\{\exp\left\{\log\log\log\frac 1{\epsilon}-16e\int_0^T\beta(s)\,ds\right\}\right\}\right\}. $$
By letting $\epsilon\to 0$, we conclude that $\alpha(t)\equiv 0$ for each $t\in (0,T_1],$
which means
$$u(t,x)\equiv 0\quad \mathrm{in } \ (0,T_1)\times \R^n.$$
The proof is therefore completed.
\end{proof}

\noindent
We now give in the following proposition an apriori estimate for solutions $u\in L^\infty(0,T;L^\infty)$
to the transport equation subject to a vanishing initial value. This estimate is the key of the proof of Theorem \ref{main}.

\begin{prop}\label{3-energy}
Let $T>0$, and assume that  $b\in L^1(0,T;W^{1,1}_{loc})$ satisfies \eqref{hyp-b-1} and
$$\div  b\in L^1(0,T;L^1\cap L^2)+L^1(0,T;L^\infty).$$
Let $u\in L^\infty(0,T;L^\infty)$ be a weak solution of
$$
\begin{cases}
\dfrac{\partial u}{\partial t}+b\cdot \nabla u=0 &  (0,T)\times \R^n,\\
u(0,\cdot)=0 &   \R^n.
      \end{cases}
$$
Then $u\in L^\infty(0,T;L^2)$.
\end{prop}
\begin{proof}
Once more, we write $\div \,b=B_{1}+B_{2},$ where now $B_1\in L^1(0,T;L^1\cap L^2)$ and $B_2\in L^1(0,T;L^\infty)$. Let us begin with the following backward transport problem,
$$
\begin{cases}
\dfrac{\partial v}{\partial t}+b\cdot \nabla v+B_2 \,v=0 &  (0,T_0)\times \R^n,\\
v(T_0,\cdot)= \chi_K\,u(T_0,\cdot)  &  \R^n,
\end{cases}
$$
where $T_0\in (0,T]$, and $\chi_K$ is the characteristic function of an arbitrary compact set $K\subset \R^n$. By Proposition \ref{p-exist}, we see that this problem admits a solution
$v\in L^\infty(0,T_0;L^1\cap L^\infty)$, because certainly $\chi_K\,u(T_0,\cdot)$ belongs to  $L^1\cap L^\infty$. Now we regularize both backward and forward problems with the help of a mollifier $0\le\rho\in \cC^\infty_c(\R^n)$, and obtain two regularized problems,
$$
\begin{cases}
    \dfrac{\partial u_\ez}{\partial t}+b\cdot \nabla u_\ez=r_{u,\ez} & (0,T)\times\R^n ,\\
      u_\ez(0,\cdot)=0 &   \R^n;
      \end{cases}
      $$
and
$$
\begin{cases}
\dfrac{\partial v_\ez}{\partial t}+b\cdot \nabla v_\ez+B_2\, v_\ez=r_{v,\ez} &   (0,T_0)\times\R^n,\\
      v_\ez(T_0,\cdot)=(\chi_K\,u(T_0,\cdot))\ast\rho_\ez &   \R^n;
\end{cases}$$
where $u_\ez=u\ast \rho_\ez$, $v_\ez=v\ast \rho_\ez$,  and $r_{u,\ez}$, $r_{v,\ez}$ converge to 0 in the $L^1(0,T;L^1_{loc})$ sense, see Lemma \ref{regularize}. Choose now $\psi_R\in \cC_c^\infty(\R^n)$ with $\psi_R\equiv 1$ on $B(0,R)$, $\supp \psi_R\subset B(0,2R)$ and $|\nabla \psi_R|\le C/R$. Then we multiply the first equation by $v_\ez\psi_R$, and integrate over time and space. We conclude that
\begin{eqnarray*}
0&&=\int_0^{T_0}\int_\R^n \left(\frac{\partial u_\ez}{\partial t}+b\cdot \nabla u_\ez-r_{u,\ez}\right)v_\ez\psi_R\,dx\,dt\\
&&=\int_{\R^n} u_\epsilon(T_0,x) (\chi_K\,u(T_0,\cdot))\ast\rho_\ez(x)\,\psi_R(x)\,dx\\
&&\quad-\int_0^{T_0}\int_\R^n \left[u_\ez \left(\frac{\partial v_\ez}{\partial t}
+b\cdot \nabla v_\ez+v_\ez\div b \right)\psi_R
+u_\ez v_\ez b\cdot\nabla \psi_R+r_{u,\ez} v_\ez\psi_R\right]\,dx\,dt\\
&&=\int_{\R^n} u_\epsilon(T_0,x) (u(T_0)\chi_K)\ast\rho_\ez(x)\psi_R(x)\,dx\\
&&\quad-\int_0^{T_0}\int_{\R^n}\left[\left(u_\ez v_\ez B_1  +v_\ez r_{u,\ez}+u_\ez r_{v,\ez}\right)\psi_R+u_\ez v_\ez b\cdot \nabla \psi_R\right]\,dx\,dt.
\end{eqnarray*}
Notice that $u_\ez\in L^\infty(0,T;L^\infty)$,
$v_\ez\in L^\infty(0,T;L^1\cap L^\infty)$ for each $\epsilon>0$, and
$r_{u,\ez}, r_{v,\ez}\to 0$ as $\epsilon\to 0$ in $L^1(0,T;L^1_{loc})$ by Lemma \ref{regularize}.
Letting $\ez\to 0$ in the above equality yields that
\begin{eqnarray*}
\int_{\R^n}u^2(T_0,x)\chi_K(x)\psi_R(x)\,dx\le \int_0^{T_0}\int_{\R^n} |u v B_1|\psi_R+|u v| |b\cdot \nabla \psi_R|\,dx\,dt.
\end{eqnarray*}
Using the fact $\frac{b(t,x)}{1+|x|}\in L^1(0,T;L^1)+L^1(0,T;L^\infty)$, $uv\in L^\infty(0,T;L^1\cap L^\infty)$, and letting $R\to \infty$,
yields
\begin{eqnarray}\label{k-estimate}
\int_{\R^n} u^2(T_0,x)\chi_K(x)\,dx\le \int_0^{T_0}\int_{\R^n} |u v B_1|\,dx\,dt.
\end{eqnarray}
Denote by $M,\,\widetilde M$ the quantities  $\|u\|_{L^\infty(0,T_0;L^\infty)}$
and $\|B_1\|_{L^1(0,T;L^1)}+\|B_1\|_{L^1(0,T;L^2)}$, respectively.
 Since $v$ is a solution to the transport equation, by using \eqref{p-bdd-solution}, we see that
\begin{eqnarray*}
\|v\|_{L^\infty(0,T;L^2)}^2&&\le \left\{\|u(T_0,\cdot)\chi_K\|_{L^2}^2+M^2\int_0^{T_0}
\int_{\R^n}|B_1|\,dx\,dt\right\} \exp\left\{2\int_0^{T_0}\|B_2\|_{L^\infty}\,dt\right\},
\end{eqnarray*}
By this, the fact $T_0\in (0,T]$, and using the H\"older inequality, we get from \eqref{k-estimate} that
\begin{eqnarray*}
\int_{\R^n} u^2(T_0,x)\chi_K(x)\,dx&&\le M\int_0^{T_0}\int_{\R^n} |v| |B_1|\,dx\,dt\le M\int_0^{T_0}\|v\|_{L^2}\|B_1\|_{L^2}\,dt\\
&&\le  M\left(\int_0^{T_0}\|B_1\|_{L^2}\,dt\right) \exp\left\{\int_0^{T_0}\|B_2\|_{L^\infty}\,dt\right\}\\
&&\quad  \times\left\{\|u(T_0,\cdot)\chi_K\|_{L^2}^2+M^2\int_0^{T_0}
\int_{\R^n}|B_1|\,dx\,dt\right\}^{1/2}\nonumber\\
&& \le M\widetilde M\exp\left\{\int_0^{T}\|B_2\|_{L^\infty}\,dt\right\}\left\{\|u(T_0,\cdot)\chi_K\|_{L^2}+M\left(\widetilde M\right)^{1/2}\right\}.
\end{eqnarray*}
An application of the Young inequality gives us that
\begin{eqnarray*}
\int_{\R^n} u^2(T_0,x)\chi_K(x)\,dx&& \le 2M^2\left((\widetilde M)^{3/2}+(\widetilde M)^{2}\right) \exp\left\{2\int_0^{T}\|B_2\|_{L^\infty}\,dt\right\},
\end{eqnarray*}
where the right hand side is independent of $K$ and $T_0$. By using the Fatou Lemma, we can finally conclude that
$$\|u\|_{L^\infty(0,T;L^2)}\le 2M^2\left((\widetilde M)^{3/2}+(\widetilde M)^{2}\right) \exp\left\{2\int_0^{T}\|B_2\|_{L^\infty}\,dt\right\},$$
i.e., $u\in L^\infty(0,T;L^2),$  which completes the proof.
\end{proof}

\noindent
We can now complete the proof of Theorem \ref{main}.

\begin{proof}[Proof of Theorem \ref{main}] By \cite[Proposition 2.1]{dl89}, we know that
there exists a weak solution $u\in L^\infty(0,T;L^\infty)$. Let us prove uniqueness.
Suppose that $u\in L^\infty(0,T;L^\infty)$ is a solution to the equation
$$
  \begin{cases}
    \dfrac{\partial u}{\partial t}+b\cdot \nabla u=0 & (0,T)\times\R^n,\\
      u(0,\cdot)= 0 & \R^n.
      \end{cases}
$$
Notice that,
since $\div \,b\in L^1(0,T;\Exp(\frac{L}{\log L}))+ L^1(0,T;L^\infty)$, we have
$$\div \,b\in L^1(0,T;L^1\cap L^2)+ L^1(0,T;L^\infty).$$
By Proposition \ref{3-energy}, we see that $u\in L^\infty(0,T;L^2)$, and so
$u\in L^\infty(0,T;L^2\cap L^\infty)$.
Then since $\div \,b\in L^1(0,T;\Exp(\frac{L}{\log L}))+ L^1(0,T;L^\infty)$, we can apply
Proposition \ref{p-unique} and obtain that such weak solution $u\in L^\infty(0,T; L^2\cap L^\infty)$ is unique. It is obvious that $u\equiv 0$ is such a unique weak solution. The proof is completed.
\end{proof}

\section{Stability}

\noindent
In this section, we prove Theorem \ref{quant}, and provide some stability result for the transport equation for vector fields having exponentially integrable divergence.

\begin{thm}\label{stability}
Let $T>0$. Assume that  $b\in L^1(0,T;W^{1,1}_{loc})$ satisfies \eqref{hyp-b-1} and \eqref{hyp-b-2}.
Suppose that $u_0\in L^\infty$ and $\{u_0^k\}_k\in L^\infty$ have uniform upper bound in $L^\infty$,
and $u^k_0-u_0\to 0$ as $k\to\infty$ in $L^p$. Let $u,\,u^k\in L^\infty(0,T;L^\infty)$ be the solutions of the transport equation
$$\frac{\partial u}{\partial t}+b\cdot \nabla u=0 \ \mathrm{in }\ (0,T)\times\R^n,$$
subject to the initial values $u_0$, $u_0^k$, respectively. Then
$$u^k-u\to 0,\quad \mathrm{in}\, L^\infty(0,T;L^p)$$
as $k\to \infty$.
\end{thm}
\begin{proof} {\bf Step 1.} Let $v_0^k=u^k_0-u_0$, and $v^k=u^k-u$ for each $k$. Denote by
$$M=\sup_k \|v^k_0\|_{L^\infty}=\sup_k \|u^k_0-u_0\|_{L^\infty}.$$
Notice that $v^k$ is the unique solution in $L^\infty(0,T; L^\infty)$ of
$$
\begin{cases}
    \dfrac{\partial v^k}{\partial t}+b\cdot \nabla v^k=0 &  (0,T)\times\R^n;\\
      v^k(0,\cdot)= v^k_0 &   \R^n.
      \end{cases}
$$
On the other hand, by Proposition \ref{p-unique}, there exists a unique solution $\tilde v^k\in L^\infty(0,T;L^p\cap L^\infty)$, since $v^k_0\in L^p\cap L^\infty$. By the uniqueness, we see that  $v^k=\tilde v^k\in L^\infty(0,T;L^\infty)$, and hence, $v^k\in L^\infty(0,T;L^p\cap L^\infty)$. Write $\div \,b=B_{1}+B_{2},$ where $B_1\in L^1(0,T;\Exp(\frac{L}{\log L}))$ and $B_2\in L^1(0,T;L^\infty)$. Then, from the estimates \eqref{infty-bdd-solution} and \eqref{p-bdd-solution} of Proposition \ref{p-exist}, we see that
\begin{equation}\label{bdd-stability-2}
\|v^k\|_{L^\infty(0,T;L^\infty)}\leq \|v_0^k\|_{L^\infty}\le M,
\end{equation}
and
\begin{equation}\label{p-bdd-stability}
\|v^k\|_{L^\infty(0,T;L^p)}^p
\leq \left\{\|v^k_{0}\|_{L^p}^p+M^p\int_0^T\|B_1\|_{L^1}\,dt\right\}\exp\left\{\int_0^T\|B_2\|_{L^\infty}\,dt\right\}.
\end{equation}
For each $R>0$, let $\psi_R\in \cC^\infty_c(\R^n)$ be as in \eqref{cutoff}. Arguing as in \eqref{energy-control-3} we see that
$$
\int_{\R^n}(|v^k(t_1)|^p-|v^k(t_0)|^p)\psi_R \,dx\\
=\int^{t_1}_{t_0}\int_{\R^n}\div  b\cdot|v^k|^p\psi_R\,dx\,ds+\int^{t_1}_{t_0}\int_{\R^n}
b\cdot \nabla \psi_R|v^k|^p\,dx\,ds.
$$
for any $0\le t_0<t_1\le T$. For the second term above, we use $|v^k|^p\in L^\infty(0,T;L^1\cap L^\infty)$ and $\frac{b(t,x)}{1+|x|}\in L^1(0,T;L^1)+L^1(0,T;L^\infty)$ to see that
$$\lim_{R\to\infty}\left|\int^{t_1}_{t_0}\int_{\R^n} b\cdot \nabla \psi_R|v^k|^p\,dx\,ds\right|\le \lim_{R\to\infty}
\left|\int^{t_1}_{t_0}\int_{B(0,2R)\setminus B(0,R)} \frac{|b(s,x)|}{1+|x|}|v^k|^p\,dx\,ds\right|=0.$$
Thus, with the help of \eqref{energy-control-3} we get that
\begin{equation}\label{3.3}
\aligned
\|v^k(t_1)\|^p_{L^p}
&\leq \|v^k(t_0)\|^p_{L^p}+ \int^{t_1}_{t_0}\int_{\R^n}|\div  b| |v^k|^p\,dx\,ds\\
&\leq \|v^k(t_0)\|^p_{L^p} +\int^{t_1}_{t_0}\int_{\R^n}|B_1| |v^k|^p\,dx\,ds+
\int^{t_1}_{t_0}\left\|B_2\right\|_{L^\infty}\|v^k\|^p_{L^p}\,ds.
\endaligned\end{equation}
Notice that  $B_1\in L^1(0,T;\Exp(\frac{L}{\log L}))\subset L^1(0,T;L^1)$, and $B_2\in L^1(0,T;L^\infty)$. This, together with \eqref{p-bdd-stability}, \eqref{bdd-stability-2} and the fact that $\|v_0^k\|_{L^p}\to 0$, further implies that there exist $i, K_1\in \N$ and $0=T_0<T_1<\cdots<T_{i}<T_{i+1}=T$ such that
\begin{eqnarray}\label{3.4}
&& \int_{T_j}^{T_{j+1}}\int_{\R^n}|B_1| |v^k|^p\,dx\,ds+
\int_{T_j}^{T_{j+1}}\left\|B_2\right\|_{L^\infty}\,\|v^k\|_{L^p}^p\,ds\le \frac12\exp(-e^{e+M}).
\end{eqnarray}
for each $0\le j\le i$ and $k\ge K_1$. For convenience, in what follows we denote by $\alpha_k(t)$, $\beta_1(t)$, $\beta_2(t)$ the quantities $\|v^k(t,\cdot)\|_{L^p}^p$, $\|B_1(t,\cdot)\|_{\Exp(\frac{L}{\log L})}$ and $\|B_2(t,\cdot)\|_{L^\infty}$, respectively. Denote also $\beta_1(t)+\beta_2(t)$ by $\beta(t)$.\\
\\
{\bf Step 2.} Let us introduce a continuous function as, for each $s\in (0,T]$,
$$\alpha^\ast(s)=\exp\left\{-\exp\left\{\exp\left\{\log\log\log\frac 1{\epsilon}-16e\int_0^s\beta(s)\,ds\right\}\right\}\right\},$$
where $\epsilon>0$ is small enough so that
$$\alpha^\ast(T)=\exp\left\{-\exp\left\{\exp\left\{\log\log\log\frac 1{\epsilon}-16e\int_0^T\beta(s)\,ds\right\}\right\}\right\}<\frac12\exp(-e^{e+M}).$$
From the definition, we see that $\alpha^\ast$ is Lipschitz smooth and increasing on $[0,T]$.\\
\\
{\bf Step 3.} Using again that $\|v_0^k\|_{L^p}\to 0$, we find that there exists $K_\ez\ge K_1$ such that
$$\alpha_k(0)=\|v_0^k\|_{L^p}^p<\ez,\hspace{1cm}\text{whenever }k\geq K_\ez.$$
Using this fact, and equations \eqref{3.3} and \eqref{3.4}, we conclude that
$$\aligned
\alpha_k(t)
&< \ez+\int^{T_1}_{0}\int_{\R^n}|B_1| |v^k|^p\,dx\,ds+\int^{T_1}_{0}\left\|B_2\right\|_{L^\infty}\|v^k\|_{L^p}^p\,ds\\
&< \exp(-e^{e+M}),
\endaligned$$
for each $t\in (0,T_1]$. Therefore, if $t\in (0,T_1]$ and $k\ge K_\ez$ we have
$$\log(e+M)\le |\log \alpha_k(t)|=\log\frac{1}{\alpha_k(t)}$$
and
$$\log\log(e^e+M)\le |\log(|\log \alpha_k(t)|)|=\log\log\frac{1}{\alpha_k(t)}.$$
By using the first part of Lemma \ref{duality}, we find for $t\in (0,T_1]$ and $k\ge K_\ez$ that
\begin{eqnarray*}
&&\int_{\R^n}|B_1| |v^k|^p\,dx\le 2\|B_1\|_{\Exp(\frac{L}{\log L})}\||v^k|^p\|_{L\log L\log\log L}\\
&&\quad\quad\le 4e\|B_1\|_{\Exp(\frac{L}{\log L})}\,\alpha_k(t)\left[\log(e+M)+|\log \alpha_k(t)|\right]
\left[\log\log(e^e+M)+|\log|\log(\alpha_k(t)||\right]\\
&&\quad\quad \le 16e\|B_1\|_{\Exp(\frac{L}{\log L})}\,\alpha_k(t)\log \left(\frac 1{\alpha_k(t)}\right)
\log\log\left(\frac{1}{ \alpha_k(t)}\right),
\end{eqnarray*}
 which together with \eqref{3.3} yields
$$\aligned
\alpha_k(t)
&\leq 16e\int_{0}^t \beta_1(s) \alpha_k(s)\log \frac1{\alpha_k(s)}\log\log\frac1{\alpha_k(s)}+\beta_2(s)\alpha_k(s)\,ds+\|v^k_0\|_{L^p}^p\\
&\le 16e\int_{0}^t \beta(s) \alpha_k(s)\log \frac1{\alpha_k(s)}\log\log\frac1{\alpha_k(s)}\,ds+\ez.
\endaligned$$
Notice that by the definition of $\alpha^\ast(t)$, we find that
$$\alpha^\ast(t)=\epsilon+16e\int_{0}^t \beta(s)\alpha^\ast(s)\log \frac 1{\alpha^\ast(s)}\log\log \frac 1{\alpha^\ast(s)}\,ds.$$
Then for each $t\in [0,T_1]$ and $k\ge K_\ez$, by the fact $\alpha_k(t)<e^{-e}$, and the function on $t|\log t||\log(|\log t|)|$ is increasing on $[0,e^{-e}]$, we see that
$$0\le \alpha_k(t)\le \alpha^\ast(t).$$
This together with the fact that $\alpha^\ast(t)$ is increasing on $[0,T_1]$ implies
$$0\le \alpha_k(t)\le \alpha^\ast(T_1)<\frac 12\exp\left\{-\exp\{e+M\}\right\}.$$
{\bf Step 4.} We can now iterate the approach to get the desired estimates. By the choice of $T_i$ (see \eqref{3.4}) and Step 3, we see that for each $t\in (T_1,T_2]$ and $k\ge K_\ez$,
$$\aligned
\alpha_k(t)&\le \alpha_k(T_1)+\int^{T_2}_{T_1}\int_{\R^n}|B_1| |v^k|^p\,dx\,ds+\int^{T_2}_{T_1}\left\|B_2\right\|_{L^\infty}\|v^k\|_{L^p}^p\,ds\\
&< \exp\left\{-\exp\{e+M\}\right\}.\endaligned
$$
Hence, for all $t\in (0,T_2]$ and $k\ge K_\ez$, we have
\begin{eqnarray*}
\alpha_k(t) &&\le 16e\int^{T_2}_{0} \beta_1(s) \alpha_k(s)\log \frac1{\alpha_k(s)}\log\log\frac1{\alpha_k(s))}+\beta_2(s)\alpha_k(s)\,ds+\int_{\R^n}|v^k|^p(0,x)\,dx\\
&&\le 16e\int^{T_2}_0\beta(s) \alpha_k(s)\log \frac1{\alpha_k(s)}\log\log\frac1{\alpha_k(s)}\,ds+\alpha_k(0),
\end{eqnarray*}
and, by the definition of $\alpha^\ast$,
$$\alpha^\ast(t)=\ez+16e\int_{0}^{t} \beta(s)\alpha^\ast(s)\log \frac 1{\alpha^\ast(s)}\log\log \frac 1{\alpha^\ast(s)}\,ds.$$
Therefore, the proof of Step 3 works well for $(0,T_2]$, and hence, we see that
$$0\le\alpha_k(t)\le \alpha^\ast(t)$$
for all $t\in (0,T_2]$ and $k\ge K_\ez$. Repeating this argument $i-1$ times more, we can conclude that for all $t\in (0,T]$
and $k\ge K_\ez$, it holds
\begin{equation}\label{stability-p}
\|v^k\|^p_{L^\infty(0,T;L^p)}\le \alpha^\ast(T)=\exp\left\{-\exp\left\{\exp\left\{\log\log\log\frac 1{\epsilon}-16e\int_0^T\beta(s)\,ds\right\}\right\}\right\},
\end{equation}
which gives the desired estimate, and completes the proof.
\end{proof}

\noindent
We next prove Theorem \ref{quant}.

\begin{proof}[Proof of Theorem \ref{quant}]
Let $\ez>0$ be chosen small enough such that
$$\exp\left\{-\exp\left\{\exp\left\{\log\log\log\frac 1{\epsilon}-32e\int_0^T\beta(s)\,ds\right\}\right\}\right\}<\frac{1}{2}\exp\left\{-\exp\{e+M\}\right\}.$$
Then by \eqref{stability-p}, we know that if $\|u_0\|_{L^p}<\ez$ and $\|u_0\|_{L^\infty}<M$, then the solution
$u$ satisfies
\begin{equation*}
\|u\|^p_{L^\infty(0,T;L^p)}\le \exp\left\{-\exp\left\{\exp\left\{\log\log\log\frac 1{\epsilon}-16e\int_0^T\beta(s)\,ds\right\}\right\}\right\},
\end{equation*}
and hence,
\begin{equation}\label{quant-1}
\|u\|^p_{L^\infty(0,T;L^p)}\le \exp\left\{-\exp\left\{\exp\left\{\log\log\log\frac 1{\|u_0\|_{L^p}}-16e\int_0^T\beta(s)\,ds\right\}\right\}\right\},
\end{equation}
Notice that
\begin{eqnarray*}
&&\exp\left\{-\exp\left\{\exp\left\{\log\log\log\frac 1{\|u(T)\|^p_{L^p}}-16e\int_0^{T}\beta(s)\,ds\right\}\right\}\right\}\\
&&\quad\le \exp\left\{-\exp\left\{\exp\left\{\log\log\log\frac 1{\ez}-32e\int_0^{T}\beta(s)\,ds\right\}\right\}\right\}< \frac{1}{2}\exp\left\{-\exp\{e+M\}\right\}.
\end{eqnarray*}
Therefore, by considering the backward equation and using the estimate \eqref{stability-p}  again, we obtain
\begin{equation*}
 \|u\|^p_{L^\infty(0,T;L^p)}\le
 \exp\left\{-\exp\left\{\exp\left\{\log\log\log\frac 1{\|u(T)\|^p_{L^p}}-16e\int_0^T\beta(s)\,ds\right\}\right\}\right\},
\end{equation*}
which implies that
  \begin{equation*}
 \|u_0\|^p_{L^p}\le
 \exp\left\{-\exp\left\{\exp\left\{\log\log\log\frac 1{\|u\|^p_{L^\infty(0,T;L^p)}}-16e\int_0^T\beta(s)\,ds\right\}\right\}\right\}.
\end{equation*}
Combining this and \eqref{quant-1} we get the desired estimate and complete the proof.
\end{proof}

\noindent
Similarly, by considering vector fields with exponentially integrable divergence, we arrive at the following quantitative estimate. Since the proof is rather identical to the above theorem, we will skip the proof.

\begin{coro}\label{quant-ex}
Let $T,M>0$ and $1\le p<\infty$. Suppose that $b\in L^1(0,T;W^{1,1}_{loc})$ satisfies \eqref{hyp-b-1} and
$$\div b\in L^1(0,T;L^\infty)+L^1(0,T;\Exp L).$$
Then there exists $\epsilon>0$ with the following property: if $u_0\in L^p\cap L^\infty$ satisfies $\|u_0\|_{L^\infty}\le M$ and $\|u_0\|^p_{L^p}< \ez$, then the problem \eqref{tran-eq} has exactly one weak solution $u$ satisfying
$$\left|\log\log\frac1{\|u\|_{L^\infty(0,T;L^p)}^p}-\log\log\frac1{\|u_0\|^p_{L^p}}\right|\leq 16e\int_0^T\beta(s)\,ds,$$
where $\div b=B_1+B_2$ and $\beta(t)=\|B_1(t,\cdot)\|_{\Exp L}+\|B_2(t,\cdot)\|_{L^\infty}$.
\end{coro}

\section{Extension to BV vector fields}

\noindent
In this section, we shall prove Theorem \ref{main-bv}. Let us begin by recalling the renormalization result of Ambrosio \cite{a04}. An $L^1$ function is said to belong to $BV$ if its first order distributional derivatives are finite Radon measures. By a $BV_{loc}$ function we mean any $L^1_{loc}$ function whose first order distributional derivatives are locally finite Radon measures. See \cite{afp00} fore more on $BV$ functions.

\begin{thm}[Ambrosio]\label{renormalized-bv}
Let $u\in L^\infty(0,T;L^\infty)$ be a solution of the Cauchy problem
$$
\begin{cases}
 \dfrac{\partial u}{\partial t}+b\cdot \nabla u+c\,u=0 & (0,T)\times \R^n,\\
u(0,\cdot)=u_0 & \R^n.
\end{cases}$$
Here $b\in L^1(0,T;BV_{loc})$ with $\div b\in L^1(0,T;L^1_{loc})$ and $c\in L^1(0,T;L^1_{loc})$. Then, for each $\beta\in \cC^1(\R)$ the composition $\beta(u)$ is a weak solution to the transport problem
$$
\begin{cases}
\frac{\partial \,\beta(u)}{\partial t}-b\cdot\nabla \beta(u)+c\,u\,\beta'(u)=0&(0,T)\times \R^n,\\
\beta(u)(0,\cdot)=\beta(u_0)&\R^n.
\end{cases}$$
\end{thm}
\begin{proof}
See the proof of \cite[Theorem 3.5]{a04}; see also \cite{cr09}.
\end{proof}
\noindent
As explained in the introduction, the following result is a kind of multiplicative property for solutions of the transport equation.

\begin{prop}\label{commutator-bv-2}
Let $T>0$, $b\in L^1(0,T;BV_{loc})$ with $\div b\in L^1(0,T;L^1_{loc})$, and $c_1,c_2\in  L^1(0,T;L^1_{loc})$.
Suppose that $u,v\in L^\infty(0,T;L^\infty)$ are solutions of the transport equation
$$\frac{\partial u}{\partial t}+b\cdot \nabla u+c_i\,u=0 \ \mathrm{in }\ (0,T)\times\R^n,$$
$i=1,2$, respectively. Then, the pointwise multiplication $u\,v$ is a solution of
$$\frac{\partial u}{\partial t}+b\cdot \nabla u+c_1\,u+c_2\,u=0 \ \mathrm{in }\ (0,T)\times\R^n.$$
\end{prop}

\begin{proof}
Let $0\le \rho\in  \cC^\infty_c(\R^n)$ be an even function, such that $\int_{\R^n}\rho\,dx=1$. For each $\ez>0$,
set $\rho_\ez(x)=\ez^{-n}\rho(x/\ez)$. We can use $\rho$ to mollify both equations, and obtain that
$$
\aligned
\dfrac{\partial \,u\ast \rho_\ez}{\partial t}+b\cdot \nabla u\ast \rho_\ez+c_1\,u\ast \rho_\ez&=r_\ez,\text{ and}\\
\dfrac{\partial\, v\ast \rho_\ez}{\partial t}+b\cdot \nabla v\ast \rho_\ez+c_2\,v\ast \rho_\ez&=s_\ez,\endaligned$$
where $$r_\ez=b\cdot \nabla u\ast \rho_\ez-(b\cdot \nabla u)\ast \rho_\ez+c_1\,u\ast \rho_\ez-(c_1\,u)\ast \rho_\ez$$
and $$s_\ez=b\cdot \nabla v\ast \rho_\ez-(b\cdot \nabla v)\ast \rho_\ez+c_2\,v\ast \rho_\ez-(c_2\,v)_\ez.$$
Therefore, we see that
$$\frac{\partial (u\ast \rho_\ez v\ast \rho_\ez)}{\partial t}+b\cdot \nabla (u\ast \rho_\ez\,\, v\ast \rho_\ez)
+(c_1+c_2)\,u\ast \rho_\ez\, v\ast \rho_\ez=v\ast \rho_\ez \,r_\ez+u\ast \rho_\ez \,s_\ez.$$
Since $u,v\in L^\infty(0,T;L^\infty)$, by using the commutator estimate from \cite[Theorem 3.2]{a04}, we see that for each compact set $K\subset \R^n$,
$$\limsup_{\ez\to0}\int_0^T\int_K|v\ast \rho_\ez \,r_\ez+u\ast \rho_\ez \,s_\ez|\,dx\,dt<\infty.$$
Letting $\ez\to 0$, we obtain that
$\frac{\partial\, (u\,v)}{\partial t}+b\cdot \nabla (u\,v)+(c_1+c_2)\,u \,v$
is a signed measure with finite total variation in $(0,T)\times K$. Then arguing as Step 2 and Step 3 of the proof of \cite[Tehorem 3.5]{a04}, we see that $u\,v$ is a solution of
$$\frac{\partial \,(u\,v)}{\partial t}+b\cdot \nabla (u\,v)+(c_1+c_2)\,u \,v=0.$$
The proof is completed.
\end{proof}

\noindent
With the aid of Theorem \ref{renormalized-bv}, we next outline the proof of Theorem \ref{main-bv}, which is similar to that of Theorem \ref{main}.

\begin{proof}[Proof of Theorem \ref{main-bv}]
The proof of existence is rather standard, and is similar to that of Proposition \ref{p-exist}, which will be omitted. Uniqueness follows by combining the following steps which are analogues of Propositions \ref{p-exist}, \ref{p-unique} and \ref{3-energy}.\\
\\
{\bf Step 1.} In this step, we show that if $p\in [1,\infty)$ and $c\in L^1(0,T;L^\infty)$, then for each $u_0\in L^\infty\cap L^p$ there is a unique weak solution $u\in L^\infty(0,T;L^p\cap L^\infty )$ of the transport problem
$$
\begin{cases}
\dfrac{\partial u}{\partial t}+b\cdot \nabla u+c\,u=0 & (0,T)\times\R^n,\\
 u(0,\cdot)=u_0 &\R^n.
\end{cases}$$
The proof of existence is the same as that of Proposition \ref{p-exist}.\\
\\
{\bf Step 2.}
In this step, we show that, if $c\in L^1(0,T;L^\infty )$, then for each $u_0\in  L^2 \cap L^\infty $, there is at most one weak solution $u\in L^\infty(0,T;L^2 \cap L^\infty )$ to the transport problem
$$
\begin{cases}
    \frac{\partial u}{\partial t}+b\cdot \nabla u+c\,u=0 &  (0,T)\times\R^n,\\
      u(0,\cdot)=u_0 &   \R^n.
 \end{cases}$$
For uniqueness, let us suppose the initial value $u_0\equiv 0$. For each $R>0$, let $\psi_R\in  \cC^\infty_c(\R^n)$ be a cut-off function as in \eqref{cutoff}. By using Theorem \ref{renormalized-bv} and integrating over time and space, we see that
\begin{eqnarray*}
&&\int_{\R^n}|u(t,x)|^2\,\psi_R(x)\,dx=\int^t_0\int_{\R^n}\left(\div  b-p\,c\right)\,|u|^2\,\psi_R\,dx\,ds+\int_0^t\int_{\R^n} b\cdot \nabla \psi_R\,|u|^2\,dx\,ds.
\end{eqnarray*}
Then the rest proof is the same as that of Proposition \ref{p-unique}.\\
\\
{\bf Step 3.} In this step, we show that if $u\in L^\infty(0,T;L^\infty )$ is a solution to the transport equation
$$
\begin{cases}
\dfrac{\partial u}{\partial t}+b\cdot \nabla u=0 & (0,T)\times\R^n,\\
 u(0,\cdot)=0 &\R^n.
\end{cases}$$
then $u\in L^\infty(0,T;L^2)$. To this end, we write $\div \,b=B_{1}+B_{2},$ where $B_1\in L^1(0,T;\Exp(\frac{L}{\log L}))$ and $B_2\in L^1(0,T;L^\infty )$. Now, we consider the following backward transport equation, given as
$$
\begin{cases}
    \dfrac{\partial v}{\partial t}+b\cdot \nabla v+B_2 v=0 & (0,T)\times \R^n,\\
      v(T_0,\cdot)= \chi_K\,u(T_0,\cdot)  &   \R^n,
      \end{cases}
$$
where $T_0\in (0,T]$, and $K$ is an arbitrary compact subset in $\R^n$. By Step 1, we see that the above admits a solution $v\in L^\infty(0,T_0,L^1\cap L^\infty)$, since $\chi_K\,u(T_0,\cdot)$ belongs to $L^1\cap L^\infty$. By Proposition \ref{commutator-bv-2}, we know that the product $u\,v$ satisfies
$$\dfrac{\partial (u\,v)}{\partial t}+b\cdot \nabla (u\,v)+B_2\,u\,v=0 \ \mathrm{in }\ (0,T)\times\R^n .$$
For each $R>0$, let $\psi_R\in  \cC^\infty_c(\R^n)$ be a cut-off function as before in \eqref{cutoff}. Then we deduce that
\begin{eqnarray*}
0&&=\int_0^{T_0}\int_{\R^n} \left(\frac{\partial (u\,v)}{\partial t}+b\cdot \nabla (u\,v)+B_2\,u\,v\right)\psi_R\,dx\,dt\\
&&=\int_{\R^n} u(T_0,x)^2\,\chi_K(x)\,\psi_R(x)\,dx+\int_0^{T_0}\int_{\R^n} \left(u\,v\psi_R\,B_2-u\, v\,\psi_R¼, \div b-u\, v\, b\cdot \nabla \psi_R\right)\,dx\,dt\\
&&=\int_{\R^n} u(T_0,x)^2\,\chi_K(x)\,\psi_R(x)\,dx-\int_0^{T_0}\int_{\R^n} \left(u\,v\,\psi_R\,B_1+u \,v\, b\cdot \nabla \psi_R\right)\,dx\,dt,
\end{eqnarray*}
which implies that
\begin{eqnarray*}
\int_{\R^n} u^2(T_0,x)\chi_K(x)\psi_R(x)\,dx\le \int_0^{T_0}\int_{\R^n} |u v B_1|\psi_R+|u v| |b\cdot \nabla \psi_R|\,dx\,dt.
\end{eqnarray*}
Once more, the rest of the proof is the same of Proposition \ref{3-energy}.\\
\\
{\bf Step 4.} In this step, we finish the proof of the theorem. If $u\in L^\infty(0,T;L^\infty)$ is a solution of  \eqref{tran-eq} with initial value $u_0=0$, then from Step 3  we know that $u\in L^\infty(0,T;L^2\cap L^\infty)$. Using Step 2, we see that such a solution $u$ must be zero, which completes the proof.
\end{proof}

\section{Counterexamples}

\noindent
In this section, we give the proof of Theorem \ref{counter}, i.e., we wish to show that the condition
$$\div \,b\in L^1\left(0,T;\Exp\left(\frac{L}{\log^\gamma L}\right)\right)\hspace{1cm}\text{for some }\gamma>1$$
is not enough to guarantee the uniqueness. We only give the example in $\R^2$, which easily can be generalized to higher dimensions. \\
\\
Let us begin with recalling an example from DiPerna-Lions \cite{dl89}. Let $K$ be a Cantor set in $[0,1]$, let $g\in  \cC^\infty(\R)$ be such that $0\le g<1$ on $\R$, and $g(x)=0$ if and only if $x\in K$. For all $x\in\R$, we set
$$f(x):=\int_0^xg(t)\,dt.$$
Since $0<g(x)<1$ at points $x\in\R\setminus K$, we see that $f$ is a $\cC^\infty$ homeomorphism from $\R$ into itself.\\
\\
Denote by $\mathcal{M}$ the set of atom-free, nonnegative, finite measures on $K$. For any measure $m \in \mathcal{M}$, the equation
$$f_m(x+m([K\cap [0,x]]))=f(x),\quad x\in\R$$
defines a function $f_m:\R\to\R$. Indeed, $f_m$ is a continuously
differentiable homeomorphism whose derivative may vanish on a set of
positive length, however, using integration by substitution we still have
$$\int_\R  v(t)\,dt = \int_\R v(f_m(s))\,f_m'(s)\,ds$$
whenever $v$ is continuous and compactly supported. One now sets
$$b(x)=(1,f'(f^{-1}(x_2))),\quad \forall\,x=(x_1,x_2)\in\R^2.$$
It follows from \cite{dl89} that, for every fixed $m\in\mathcal M$, the function
$$X_m(t,x)=(x_1+t,f_m(t+f_m^{-1}(x_2))),\quad \forall\, t\in\R, \quad x=(x_1,x_2)\in\R^2,$$
satisfies
$$\frac{\partial}{\partial t}X_m(t,x)=b(X_m(t,x)).$$
We proceed now to prove Theorem \ref{counter}.

\begin{proof}[Proof of Theorem \ref{counter}] We divide the proof into the following four steps.
Based on the example of DiPerna-Lions \cite{dl89}, the main points left are to construct an
explicit smooth function $g$ and to show that $u_0$ composed with the flow is a distributional solution.\\
\\
{\bf Step 1: A minor modification of DiPerna-Lions' vector field \cite{dl89}.}
We start with $K$, $f$ and $f_m$ as introduced above. Let $\phi\in  \cC^\infty_c(\R)$ be such that $0\le\phi\le 1$ and $\supp\phi\subset[-1,2]$, which equals one on $[0,1]$. We choose the vector field $\tilde b$ as
$$\tilde b(x):=(0,\phi(x_1)f'(f^{-1}(x_2))),\hspace{.3cm}x=(x_1,x_2)\in \R^2.$$
As $g\in  \cC^\infty(\R)$ and $0\le g<1$, it follows readily that $\tilde b\in L^1(0,T;W^{1,1}_{loc}(\R^2))$ and
$$|\tilde b(x)|\in L^\infty(0,T;L^\infty(\R^2)).$$
For each $m\in \mathcal{M}$, let
$$\widetilde X_m(t,x)=(x_1,f_m(t\,\phi(x_1)+f_m^{-1}(x_2))),\quad  t\in\R, \quad x=(x_1,x_2)\in\R^2.$$
Then we see that
$$\frac{\partial}{\partial t}\widetilde X_m(t,x)=(0,\phi(x_1)f_m'(t\,\phi(x_1)+f_m^{-1}(x_2)))=(0,\phi((\widetilde X_m)_1)f_m'(f_m^{-1}((\widetilde X_m)_2)).$$
From \cite{dl89}, we know that, for each $t\in\R$,
$$f'_m(f_m^{-1}(t))=f'(f^{-1}(t)),$$
which, together with the above equality, yields that
$$\frac{\partial}{\partial t}\widetilde X_m(t,x)=(0,\phi((\widetilde X_m)_1)f'(f^{-1}((\widetilde X_m)_2))=\tilde b(\widetilde X_m).$$
Hence, the initial value problem $\dot{X}=\tilde b(X), X(0)=x$ admits $X(t)=\widetilde X_m(t,x)$ as a solution, for every fixed $m\in\mathcal{M}$. \\
\\
{\bf Step 2: Construction of a precise function $g$.} In order to check the integrability of $\div\tilde{b}$, we need to describe $g$ explicitly. To do this, we now fix $K$ to be a one third Cantor set on $[0,1]$. By $\{C_{k_j}\}_{j=1}^{2^{k-1}}$ we denote the collection of open sets removed in the $k$-th generation,
and $\{y_{k_j}\}_{j=1}^{2^{k-1}}$ be their centers.
For each $C_{k_j}$ we associate it with a smooth function as
$$\psi_{k_j}(x):=
\left\{ \begin{array}{ll}
   \exp\left\{-\exp\left\{\exp\left\{\frac{1}{\frac{1}{(2\cdot 3^k)^2}-(x-y_{k_j})^2}\right\}\right\}\right\} & \ x\in C_{k_j};\\
     0 & \ x\in \R\setminus C_{k_j}.
      \end{array}
      \right.
      $$
We next choose the function for $(-\infty,0)$ and $(1,\infty)$ as
$$g_1(x):=
\left\{ \begin{array}{ll}
   \exp\left\{-\exp\left\{\exp\left\{\frac{1}{x^{2}}\right\}\right\}\right\} & \ x\in (-\infty,0);\\
   \exp\left\{-\exp\left\{\exp\left\{\frac{1}{\left(x-1\right)^{2}}\right\}\right\}\right\}&\ x\in (1,\infty);\\
     0 & \ x\in \R\setminus [0,1].
      \end{array}
      \right.
      $$
It is obvious that $\psi_{k_j}, \,g_1\in  \cC^\infty(\R)$.  Now we set
 $$g(x)=\sum_{k\ge 1}\sum_{j=1}^{2^{k-1}}\psi_{k_j}(x)+g_1(x).$$
 It is readily seen that $g$ is smooth on $\R$, $g(x)=0$ for each $x\in
 K$ and $0<g<1$ for each $x\in\R\setminus K$.
 Therefore, $g$ satisfies the requirements from the example of DiPerna-Lions \cite{dl89} as recalled above.\\
\\
{\bf Step 3: Checking the integrability of $\div \tilde b$}. We will prove now that $\div\tilde{b}\in L^1(0,T;\Exp(\frac{L}{\log^\gamma L}))$ whenever $\gamma>1$.  Notice that, for $1<\gamma_1<\gamma_2<\infty$, it holds that
$$\exp\left(\frac{t}{(\log^+t)^{\gamma_2}}\right)\le \exp\left(\frac{t}{(\log^+t)^{\gamma_1}}\right)$$
for each $t\geq 0$. Therefore, we only need to show that $\div \tilde b\in L^1(0,T; \Exp(\frac{L}{\log^\gamma L}))$ for each $\gamma>1$ close to one. Let us fix  $\gamma\in (1,2)$.\\
\\
Notice that the function $t\mapsto \frac{t}{(\log^+t)^\gamma}$ is not monotonic. Indeed, it is increasing on $[0,e]$ and $[e^\gamma,\infty)$ and decreasing on $[e,e^\gamma]$. However, it is not hard to see that if $0<t<s<\infty$, then
$$\frac{t}{(\log^+t)^\gamma}\le \frac{e\gamma^\gamma}{e^\gamma}\frac{s}{(\log^+s)\gamma}.$$
A direct calculation shows that
$$\div \,\tilde{b}(x)=\phi(x_1)\frac{g'(f^{-1}(x_2))}{g(f^{-1}(x_2))},\quad\forall x=(x_1,x_2)\in\R^2.$$
Recall that  $\phi\in  \cC^\infty_c(\R)$, $0\le\phi\le 1$, $\supp
\phi\subset[-1,2]$, and $\phi=1$ on $[0,1]$. Therefore,
\begin{equation}\label{4.1}
\div \,\tilde{b}(x)\equiv 0, \quad \forall x\in (\R\setminus(-1,2))\times \R,
\end{equation}
and
\begin{eqnarray}\label{4.2}
&&\int_{\R^2}\left[\exp\left\{
\frac{C\left|\div \,\tilde{b}(x)\right|}
{\left[\log^+\left(\left|C\div \,\tilde{b}(x)\right|\right)\right]^\gamma}\right\}-1\right]\,dx\nonumber\\
&&\quad=\int_{\R^2}\left[\exp\left\{
\frac{C\left|\phi(x_1)\frac{g'(f^{-1}(x_2))}{g(f^{-1}(x_2))}\right|}
{\left[\log^+\left(\left|C\phi(x_1)\frac{g'(f^{-1}(x_2))}{g(f^{-1}(x_2))}\right|\right)\right]^\gamma}\right\}-1\right]
\,dx_1\,dx_2\nonumber\\
&&\quad=
\int_{[-1,2]\times \R}\left[\exp\left\{
\frac{C\left|\phi(x_1)\frac{g'(t)}{g(t)}\right|}
{\left[\log^+\left(\left|C\phi(x_1)\frac{g'(t)}{g(t)}\right|\right)\right]^\gamma}\right\}-1\right]g(t)\,dx_1\,dt\nonumber\\
&&\quad\le \int_{\R}3\left[\exp\left\{
\frac{Ce^{1-\gamma}\gamma^\gamma\left|\frac{g'(t)}{g(t)}\right|}
{\left[\log^+\left(\left|C\frac{g'(t)}{g(t)}\right|\right)\right]^\gamma}\right\}-1\right]g(t)\,dt.
\end{eqnarray}
By the above inequality, in order to show
$\div \tilde b\in L^1(0,T;\Exp(\frac{L}{\log^\gamma L}))$, it is sufficient to show that
\begin{eqnarray}\label{4.3}
\int_{\R}\left[\exp\left\{
\frac{Ce^{1-\gamma}\gamma^\gamma\left|\frac{g'(t)}{g(t)}\right|}
{\left[\log^+\left(\left|C\frac{g'(t)}{g(t)}\right|\right)\right]^\gamma}\right\}-1\right]g(t)\,dt<\infty,
\end{eqnarray}
for some $C>0$. \\
%
%
\\
{\bf Claim 1:} For $A=\frac{(\gamma-1)^2}{2^3}$, one has
\begin{eqnarray}\label{4.4}
\int_{\R\setminus[-1,2]}\left[\exp\left\{
\frac{A e^{1-\gamma}\gamma^\gamma\left|\frac{g'(t)}{g(t)}\right|}
{\left[\log^+\left(\left|A \frac{g'(t)}{g(t)}\right|\right)\right]^\gamma}\right\}-1\right]g(t)\,dt<\infty.
\end{eqnarray}
By symmetry of the function $g$ on $(-\infty,-1)\cup (2,\infty)$ and the fact $0\le g<1$, we only need to show that
\begin{equation}\label{4.5}
\int_{(-\infty,-1)}\left[\exp\left\{
\frac{A e^{1-\gamma}\gamma^\gamma\left|\frac{g'(t)}{g(t)}\right|}
{\left[\log^+\left(\left|A e^{1-\gamma}\gamma^\gamma\frac{g'(t)}{g(t)}\right|\right)\right]^\gamma}\right\}-1\right]\,dt<\infty.
\end{equation}
By noticing that
$$\frac{g'(t)}{g(t)}=\exp\left\{\exp\left\{\frac{1}{t^2}\right\}\right\}\exp\left\{\frac{1}{t^2}\right\}\frac{2}{t^3},\quad \mathrm{if}\, t\in (-\infty,-1),$$
and
$$ \log^+\left(\left|A \frac{g'(t)}{g(t)}\right|\right)\ge 1,$$
we conclude by using the Taylor expansion that
\begin{eqnarray*}
\int_{(-\infty,-1)}\left[\exp\left\{
\frac{A e^{1-\gamma}\gamma^\gamma\left|\frac{g'(t)}{g(t)}\right|}
{\left[\log^+\left(\left|A \frac{g'(t)}{g(t)}\right|\right)\right]^\gamma}\right\}-1\right]\,dt&&\le
\int_{(-\infty,-1)}\left[\exp\left\{
\frac{\widetilde A }{t^3}\right\}-1\right]\,dt\\
&&\le \int_{(-\infty,-1)} \sum_{l=1}^\infty \frac{(\widetilde A )^{l}}{l!}\frac{1}{t^{3l}}\,dt= \sum_{l=1}^\infty \frac{(\widetilde A )^{l}}{l!}\frac{1}{3l-1}<\infty,
\end{eqnarray*}
where $\widetilde A =2A e^{1-\gamma}\gamma^\gamma \exp\{e^{1+e}\}$.
This implies \eqref{4.5} and hence, \eqref{4.4} holds.\\
\\
{\bf Claim 2.} If $A $ is as above, then
\begin{eqnarray}\label{4.6}
\int_{[-1,2]}\left[\exp\left\{
\frac{A e^{1-\gamma}\gamma^\gamma\left|\frac{g'(t)}{g(t)}\right|}
{\left[\log^+\left(\left|A \frac{g'(t)}{g(t)}\right|\right)\right]^\gamma}\right\}-1\right]g(t)\,dt<\infty.
\end{eqnarray}
Since $0\le g<1$, the above inequality will follow from
\begin{eqnarray}\label{4.7}
\int_{[-1,2]}\exp\left\{
\frac{A e^{1-\gamma}\gamma^\gamma\left|\frac{g'(t)}{g(t)}\right|}
{\left[\log^+\left(\left|A \frac{g'(t)}{g(t)}\right|\right)\right]^\gamma}\right\}g(t)\,dt<\infty.
\end{eqnarray}
Notice that, for each $x\in C_{k_j}$,
\begin{eqnarray*}
\frac{g'(x)}{g(x)}=-\exp\left\{\exp\left\{\frac{1}{\frac{1}{(2\cdot 3^k)^2}-(x-y_{k_j})^2}\right\}\right\}\exp\left\{\frac{1}{\frac{1}{(2\cdot 3^k)^2}-(x-y_{k_j})^2}\right\}
\frac{2(x-y_{k_j})}{[\frac{1}{(2\cdot 3^k)^2}-(x-y_{k_j})^2]^2},
\end{eqnarray*}
and hence,
\begin{eqnarray*}
\frac{|g'(x)|}{g(x)}&&\le \exp\left\{\exp\left\{\frac{1}{\frac{1}{(2\cdot 3^k)^2}-(x-y_{k_j})^2}\right\}\right\}\exp\left\{\frac{1}{\frac{1}{(2\cdot 3^k)^2}-(x-y_{k_j})^2}\right\}
\frac{\frac{1}{ 3^k}}{[\frac{1}{(2\cdot 3^k)^2}-(x-y_{k_j})^2]^2}\\
&&\le \exp\left\{\exp\left\{\frac{1}{\frac{1}{(2\cdot 3^k)^2}-(x-y_{k_j})^2}\right\}\right\}\exp\left\{\frac{\frac{1+\gamma}{2}}{\frac{1}{(2\cdot 3^k)^2}-(x-y_{k_j})^2}\right\}
\frac{2}{3^k}\frac{2^2}{(\gamma-1)^2}.
\end{eqnarray*}
Notice that  the function  $\frac{t}{(\log^+t)^\gamma}$ is increasing on $(0,e)\cup(e^\gamma,\infty)$
and decreasing on $(e,e^\gamma)$.
If $A \frac{|g'(x)|}{g(x)}<e^\gamma$,
then
\begin{equation}\label{g-est-1}
\frac{A \frac{|g'(x)|}{g(x)}}{\left[\log^+\left(\left|A \frac{g'(x)}{g(x)}\right|\right)\right]^\gamma}\le e^\gamma;
\end{equation}
while for $A \frac{|g'(x)|}{g(x)}\ge e^\gamma$, by the choice of $A $, we have
\begin{eqnarray}\label{g-est-2}
\frac{A \frac{|g'(x)|}{g(x)}}{\left[\log^+\left(\left|A \frac{g'(x)}{g(x)}\right|\right)\right]^\gamma}
&&\le\frac{A  \exp\left\{\exp\left\{\frac{1}{\frac{1}{(2\cdot 3^k)^2}-(x-y_{k_j})^2}\right\}\right\}
\exp\left\{\frac{\frac{1+\gamma}{2}}{\frac{1}{(2\cdot 3^k)^2}-(x-y_{k_j})^2}\right\}
\frac{2}{3^k}\frac{2^2}{(\gamma-1)^2}}{\left[\exp\left\{\frac{1}{\frac{1}{(2\cdot 3^k)^2}-(x-y_{k_j})^2}\right\}
+\frac{\frac{1+\gamma}{2}}{\frac{1}{(2\cdot 3^k)^2}-(x-y_{k_j})^2}+\log(A \frac{2}{3^k}\frac{2^2}{(\gamma-1)^2})\right]^\gamma}\nonumber\\
&&\le \frac{1}{3^k}\frac{\exp\left\{\exp\left\{\frac{1}{\frac{1}{(2\cdot 3^k)^2}-(x-y_{k_j})^2}\right\}\right\}
\exp\left\{\frac{\frac{1+\gamma}{2}}{\frac{1}{(2\cdot 3^k)^2}-(x-y_{k_j})^2}\right\}
}{\left[\exp\left\{\frac{1}{\frac{1}{(2\cdot 3^k)^2}-(x-y_{k_j})^2}\right\}+(2\cdot 3^k)^2+\log(\frac{1}{3^k})
\right]^\gamma}\nonumber\\
&&\le  \frac{1}{3^k}
\exp\left\{\exp\left\{\frac{1}{\frac{1}{(2\cdot 3^k)^2}-(x-y_{k_j})^2}\right\}\right\}
\exp\left\{\frac{\frac{1-\gamma}{2}}{\frac{1}{(2\cdot 3^k)^2}-(x-y_{k_j})^2}\right\}\nonumber\\
&&\le  \frac{1}{3^k}
\exp\left\{\exp\left\{\frac{1}{\frac{1}{(2\cdot 3^k)^2}-(x-y_{k_j})^2}\right\}\right\}
\exp\left\{{2\cdot 3^{2k}(1-\gamma)}\right\}\nonumber\\
&&\le  \frac{1}{3}
\exp\left\{\exp\left\{\frac{1}{\frac{1}{(2\cdot 3^k)^2}-(x-y_{k_j})^2}\right\}\right\}.
\end{eqnarray}
Combining \eqref{g-est-1} and \eqref{g-est-2}, we deduce that for each $x\in C_{k_j}$,
\begin{eqnarray}\label{4.8}
\exp\left\{\frac{A e^{1-\gamma}\gamma^\gamma\frac{|g'(x)|}{g(x)}}{\left[\log^+\left(\left|A \frac{g'(x)}{g(x)}\right|\right)\right]^\gamma}\right\}g(x)&&
\le \exp\left\{e\gamma^\gamma-\left[1-\frac{e^{1-\gamma}\gamma^\gamma}{3}\right]
\exp\left\{\exp\left\{\frac{1}{\frac{1}{(2\cdot 3^k)^2}-(x-y_{k_j})^2}\right\}\right\}\right\}\nonumber\\
&&\le \exp\left\{e\gamma^\gamma\right\},
\end{eqnarray}
since by assumption $1<e^{1-\gamma}\gamma^\gamma<e$.
Indeed, from \eqref{g-est-2} and \eqref{4.8}, we can further see that
the function
$$\exp\left\{\frac{A e^{1-\gamma}\gamma^\gamma\frac{|g'(x)|}{g(x)}}{\left[\log^+\left(\left|A \frac{g'(x)}{g(x)}\right|\right)\right]^\gamma}\right\}g(x)$$
is smooth  on $\overline{C_{k_j}}$, and equals $0$ on the boundary of $C_{k_j}$.\\
\\
On the other hand, notice that for each $x\in [-1,0)$, it holds that
\begin{eqnarray*}
\frac{|g'(x)|}{g(x)}=\exp\left\{\exp\left\{\frac{1}{x^2}\right\}\right\}\exp\left\{\frac{1}{x^2}\right\}
\frac{2}{|x|^3}\le \exp\left\{\exp\left\{\frac{1}{x^2}\right\}\right\}\exp\left\{\frac{1+\gamma}{2x^2}\right\}\frac{2\cdot 2^2}{(\gamma-1)^2}.
\end{eqnarray*}
If $A \frac{|g'(x)|}{g(x)}<e^\gamma$,
then
\begin{equation*}
\frac{A \frac{|g'(x)|}{g(x)}}{\left[\log^+\left(\left|A \frac{g'(x)}{g(x)}\right|\right)\right]^\gamma}\le e^\gamma,
\end{equation*}
while for $A \frac{|g'(x)|}{g(x)}\ge e^\gamma$,
\begin{eqnarray*}
\exp\left\{\frac{A e^{1-\gamma}\gamma^\gamma\frac{|g'(x)|}{g(x)}}{\left[\log^+\left(\left|A \frac{g'(x)}{g(x)}\right|\right)\right]^\gamma}\right\}g(x)&&
\le \exp\left\{e^{1-\gamma}\gamma^\gamma\frac{\exp\left\{\exp\left\{\frac{1}{x^2}\right\}\right\}\exp\left\{\frac{1+\gamma}{2x^2}\right\}}
{\exp\left\{\frac{\gamma}{x^2}\right\}}\right\}g(x)\\
&&\le  \exp\left\{\exp\left\{\exp\left\{\frac{1}{x^2}\right\}\right\}
\left[e^{1-\gamma}\gamma^\gamma\exp\left\{\frac{1-\gamma}{2x^2}\right\}-1\right]\right\}\\
&&\le  \left\{ \begin{array}{ll}
    1, & \ x\in (-\frac{\sqrt{\gamma-1}}{2},0);\\
      \exp\left\{\exp\left\{\exp\left\{\frac{4}{\gamma-1}\right\}\right\}
\exp\left\{1+\frac{1-\gamma}{2}\right\}\right\}, & \ x\in [-1,-\frac{\sqrt{\gamma-1}}{2})
      \end{array}
      \right.\\
      &&\le \exp\left\{\exp\left\{1+\exp\left\{\frac{4}{\gamma-1}\right\}\right\}
\right\},
\end{eqnarray*}
since $1<\gamma<2$. The above two inequalities imply that
\begin{eqnarray*}
\exp\left\{\frac{A e^{1-\gamma}\gamma^\gamma\frac{|g'(x)|}{g(x)}}
{\left[\log^+\left(\left|A \frac{g'(x)}{g(x)}\right|\right)\right]^\gamma}\right\}g(x)
\le \exp\left\{\exp\left\{1+\exp\left\{\frac{4}{\gamma-1}\right\}\right\}
\right\}.
\end{eqnarray*}
and, similarly, for each $x\in (1,2]$,
\begin{eqnarray*}
\exp\left\{\frac{A e^{1-\gamma}\gamma^\gamma\frac{|g'(x)|}{g(x)}}
{\left[\log^+\left(\left|A \frac{g'(x)}{g(x)}\right|\right)\right]^\gamma}\right\}g(x)\le \exp\left\{\exp\left\{1+\exp\left\{\frac{4}{\gamma-1}\right\}\right\}
\right\}.
\end{eqnarray*}
Therefore, from these two estimates together with \eqref{4.8}, we conclude that
\begin{eqnarray*}
\int_{[-1,2]}\exp\left\{\frac{A e^{1-\gamma}\gamma^\gamma\frac{|g'(x)|}{g(x)}}{\left[\log^+\left(\left|C\frac{g'(x)}{g(x)}\right|\right)\right]^\gamma}\right\}g(x)\,dx
\le 2C(\gamma)+\sum_k\sum_{j=1}^{2^{k-1}}\frac{\exp\left\{e\gamma^\gamma\right\}}{3^k}<\infty,
\end{eqnarray*}
where $C(\gamma)=\exp\left\{\exp\left\{1+\exp\left\{\frac{4}{\gamma-1}\right\}\right\}
\right\}$. This,
together with \eqref{4.7}, yields \eqref{4.6}. Combining the inequalities \eqref{4.4} and \eqref{4.6} yields
\begin{eqnarray*}
&&\int_{\R^2}\left[\exp\left\{
\frac{A \left|\div \,\tilde{b}(x)\right|}
{\left[\log^+\left(\left|A \div \,\tilde{b}(x)\right|\right)\right]^\gamma}\right\}-1\right]\,dx<\infty,
\end{eqnarray*}
via \eqref{4.2} and \eqref{4.3}, where $A =\frac{(\gamma-1)^2}{2^3}$.
Therefore, Step 3 is completed.\\
\\
\textbf{Step 4: Constructing infinitely many solutions to a transport equation}.
Let us fix a function  $u_0\in \cC^\infty_c(\R^2)$ which does not identically vanish on $(0,1)\times (0,\infty)$.
{ Then the set
$$\left\{u_m(t,x)=u_0(\widetilde X_m(t,x))\ \big|\, m\in \mathcal{M}\right\}$$
contains infinite many functions. Let us show that, }for each $m\in \mathcal{M}$,  the function $u_m(t,x)$
is a solution to the initial value problem
\begin{equation}\label{minus}
\begin{cases}
\dfrac{\partial u}{\partial t}-\tilde b\cdot \nabla u=0 & (0,T)\times\R^2,\\
      u(0,\cdot)=u_0 &  \R^2.
   \end{cases}\end{equation}
For proving this, let us choose a test function $\varphi\in \cC^\infty_c([0,T)\times\R^2)$. We fix a real number $t\in[0,T)$, and $h$ sufficiently small. We have
\begin{eqnarray*}
&&\frac
1h\int_{\R^2}\left(u_m(t+h,x)-u_m(t,x)\right)\varphi(t,x)\,dx\\
&&=\frac
1h\int_{\R^2}\left(u_0(x_1,f_m((t+h)\phi(x_1)+f_m^{-1}(x_2)))-u_0(x_1,f_m(t\phi(x_1)+f_m^{-1}(x_2)))\right)\varphi(t,x_1,x_2)\,dx_1\,dx_2\\
&&=\frac
1h\int_{\R^2}\left(u_0(x_1,f_m((t+h)\phi(x_1)+y_2))-u_0(x_1,f_m(t\phi(x_1)
+y_2))\right)\varphi(t,x_1,f_m(y_2))f_m'(y_2)\,dx_1\,dy_2,
\end{eqnarray*}
where in the last equality we used the change of variables $x_2=f_m(y_2)$. Indeed,
it is clear that $f_m$ is a continuously differentiable homeomorphism
on $\R$ (since $g(t)\to e^{-e}$ as $|t|\to\infty$). Therefore, the classical integration by substitution
$$\int_\R v(t)\,dt = \int_\R v(f_m(s))\,f'_m(s)\,ds $$
is legitimate for any continuous $v$. Now, since $u_0$ is compactly supported and smooth, and
$f_m$ is continuously differentiable, we can let $h$ tend to zero in
the above equality. We obtain
\begin{eqnarray}
&&\int_{\R^2}\frac{\partial u_m(t,x)}{\partial t}\,\varphi(t,x)\,dx\label{a1}\\
&&=\int_{\R^2}\frac{\partial u_0}{\partial
z}(x_1,z)\bigg|_{z=f_m(t\phi(x_1)+y_2)}\,f_m'(t\phi(x_1)+y_2)\,\phi(x_1)\,\varphi(t,x_1,f_m(y_2))\,f_m'(y_2)\,dx_1\,dy_2.\nonumber
\end{eqnarray}
\noindent
On the other hand, by noticing that
$$
\tilde
b(x)=(0,\phi(x_1)f'(f^{-1}(x_2)))=(0,\phi(x_1)f_m'(f_m^{-1}(x_2))),$$
we get $\tilde b\in W^{1,q}_{loc}(\R^2)$ for any finite $q$. Indeed, one clearly has
$$\frac{\partial \tilde b}{\partial x_1}\,(x_1,x_2)=(0,\phi'(x_1)f'(f^{-1}(x_2))),$$
which is continuous and bounded. Also, for each $\gamma>1$ we see that
$$\frac{\partial \tilde b}{\partial x_2}\,(x_1,x_2)=(0,\phi(x_1)\frac{g'(f^{-1}(x_2))}{g(f^{-1}(x_2))})$$
which belongs to $\Exp\left(\frac{L}{\log^\gamma L}\right)$ (see Step 3). As a consequence, we have
$$
\aligned
-\int_{\R^2}&u_m(t,x)\,\mathrm{div}(\tilde b\varphi)(t,x)\,dx\\
&=-\int_{\R^2}u_0\left(\tilde{X}_m(t,x_1,x_2))\right)\,\phi(x_1)\,\frac{\partial}{\partial
z}\bigg(f'(f^{-1}(z))\varphi(t,x_1,z)\bigg)\bigg|_{z=x_2}\,dx_1\,dx_2\\
&=-\int_{\R^2}u_0\left(\tilde{X}_m(t,x_1,f_m(y_2)))\right)\,\phi(x_1)\,\frac{\partial
}{\partial
z}\bigg(f'(f^{-1}(z))\,\varphi(t,x_1,z)\bigg)\bigg|_{z=f_m(y_2)}\,f'_m(y_2)\,dx_1\,dy_2\\
&=-\int_{\R^2}u_0\left(\tilde{X}_m(t,x_1,f_m(y_2)))\right)\,\phi(x_1)\,\frac{\partial
}{\partial
z}\bigg(f'(f^{-1}(f_m(z)))\,\varphi(t,x_1,f_m(z))\bigg)\bigg|_{z=y_2}\,dx_1\,dy_2\\
&=-\int_{\R^2}u_0\left(\tilde{X}_m(t,x_1,f_m(y_2)))\right)\,\phi(x_1)\,\frac{\partial
}{\partial
z}\bigg(f'_m(z)\,\varphi(t,x_1,f_m(z))\bigg)\bigg|_{z=y_2}\,dx_1\,dy_2\\
&=\int_{\R^2}\frac{\partial u_0 }{\partial
z}\left(x_1,z\right)\bigg|_{z=f_m(t\phi(x_1)+y_2)}f'_m(t\phi(x_1)+y_2)
\phi(x_1)\varphi(t,x_1,f_m(y_2))f_m'(y_2)\,dx_1\,dy_2,
\endaligned$$
This, together with \eqref{a1}, implies that the equality
\begin{eqnarray*}
&&\int_{\R^2}\frac{\partial u_m(t,x)}{\partial
t}\varphi(t,x)\,dx=-\int_{\R^2}u_m(t,x)\,\mathrm{div}(\tilde
b\varphi)(t,x)\,dx
\end{eqnarray*}
holds for each $\varphi\in \cC^\infty([0,T)\times\R^2)$ with compact support in $[0,T)\times \R^2$, at any time $t$. Integrating over time we obtain
\begin{eqnarray*}
&&-\int_0^T\int_{\R^2}u_m(t,x)\,\frac{\partial \varphi(t,x)}{\partial
t}\,dt\,dx- \int_{\R^2}u_m(0,x)\, \varphi(0,x)\,dx=-\int_{\R^2}u_m(t,x)\,\mathrm{div}(\tilde
b\varphi)(t,x)\,dx
\end{eqnarray*}
as desired. Thus $u_m$ is, for every $m\in\mathcal{M}$, a weak solution to \eqref{minus}, and therefore Step $4$ follows. The proof of the Theorem  \ref{counter} is concluded.
\end{proof}

\noindent
\textbf{Acknowledgments}.
The authors are grateful to Gianluca Crippa for interesting remarks which improved the paper.
Albert Clop, Joan Mateu and Joan Orobitg were partially supported by Generalitat de Catalunya (2014SGR75) and Ministerio de Econom\'\i a y Competitividad (MTM2013-44699). Albert Clop was partially supported by the Programa Ram\'on y Cajal. Renjin Jiang was partially supported by National Natural Science Foundation of China (NSFC 11301029). All authors were partially supported by Marie Curie Initial Training Network MAnET (FP7-607647).

\noindent
\textit{Albert Clop, Joan Mateu and Joan Orobitg}

\vspace{0.1cm}
\noindent
Departament de Matem\`atiques, Facultat de Ci\`encies,\\
Universitat Aut\`onoma de Barcelona\\
08193 Bellaterra (Barcelona), CATALONIA.

\vspace{0.3cm}

\noindent
\textit{Renjin Jiang}

\vspace{0.1cm}
\noindent
School of Mathematical Sciences, Beijing Normal University, Laboratory of Mathematics and Complex Systems,
Ministry of Education, 100875, Beijing, CHINA

and

\noindent
Departament de Matem\`atiques, Facultat de Ci\`encies,\\
Universitat Aut\`onoma de Barcelona\\
08193 Bellaterra (Barcelona), CATALONIA.

\vspace{0.2cm}
\noindent{\it E-mail addresses}:\\
\texttt{albertcp@mat.uab.cat}\\
\texttt{jiang@mat.uab.cat\,\&\,rejiang@bnu.edu.cn}\\
\texttt{mateu@mat.uab.cat}\\
\texttt{orobitg@mat.uab.cat}
\end{document}